\documentclass[10pt,reqno]{article}   
\usepackage{amssymb}
\usepackage{enumitem}
\usepackage{amsmath,amsbsy,amsfonts,amsthm}
\usepackage{graphicx}
\usepackage{geometry}
\usepackage{cancel}
\usepackage{hyperref}
\usepackage{cleveref}
\usepackage{color}
\usepackage[utf8]{inputenc}

\newtheorem{theorem}{Theorem}[section]
\newtheorem{lemma}[theorem]{Lemma}
\newtheorem{proposition}[theorem]{Proposition}
\newtheorem{corollary}[theorem]{Corollary}

\newtheorem{definition}[theorem]{Definition}

\begin{document}
	
	\numberwithin{equation}{section}

	\title{Regularizing Effect for a Nonlocal Maxwell--Schr\"{o}dinger System}
	\author{L. H. de Miranda and A. P. de Castro Santana}
	\date{}
	
	\maketitle{}
	\begin{abstract}
		\hspace{-.5cm} In this paper we prove existence and regularity of weak solutions for the following system
		\begin{align*}
			\begin{cases}
			&-\mbox{div}\Bigg(\bigg(\|\nabla u\|^{p}_{L^{p}}+\|\nabla v\|^{p}_{L^{p}}\bigg)|\nabla u|^{p-2}\nabla u\Bigg)  + g(x,u,v)=f \  \ \ \mbox{in} \ \Omega; \\
			&-\mbox{div}\Bigg(\bigg(\|\nabla u\|^{p}_{L^{p}}+\|\nabla v\|^{p}_{L^{p}}\bigg)|\nabla v|^{p-2}\nabla v\Bigg) = h(x,u,v) \  \ \ \ \mbox{in} \ \Omega; \\
			&u=v=0 \ \mbox{on} \ \partial\Omega.
		\end{cases}
		\end{align*}
		where $\Omega$ is an open bounded subset of $\mathbb{R}^N$,  $N>2$,  $f\in L^m(\Omega)$, where $m>1$ and $g$, $h$ are two Carath\'eodory functions, which may be non monotone. We prove that under appropriate conditions on $g$ and $h$, there is gain of Sobolev and Lebesgue regularity for the solutions of this system. 
	
    {\em Mathematics Subject Classification 2020: } Primary: 35B65, 35D30; Secondary: 35B45,35D99.
    \end{abstract}

	\section{Introduction}
	
	In this paper, we investigate the existence and regularity of positive solutions for the following  class of Maxwell-Schr\"odinger systems of Kirchhoff type
	\begin{align*}\tag{P}\label{P}
		\begin{cases}
			&-\mbox{div}\Bigg(\bigg(\|\nabla u\|^{p}_{L^{p}}+\|\nabla v\|^{p}_{L^{p}}\bigg)|\nabla u|^{p-2}\nabla u\Bigg)  + g(x,u,v)=f \  \ \ \mbox{in} \ \Omega; \\
			&-\mbox{div}\Bigg(\bigg(\|\nabla u\|^{p}_{L^{p}}+\|\nabla v\|^{p}_{L^{p}}\bigg)|\nabla v|^{p-2}\nabla v\Bigg) = h(x,u,v) \  \ \ \ \mbox{in} \ \Omega; \\
			&u=v=0 \ \mbox{on} \ \partial\Omega,
		\end{cases}
	\end{align*}	
where, from now on, $\Omega \subset \mathbb{R}^N$, is an open bounded subset, $N>2$, $1<p<N$ and $f\in L^m(\Omega)$ for $m>1$. Moreover, $g$ and $h:\Omega\times\mathbb{R}\times\mathbb{R}\to \mathbb{R}$ 
are both  Carath\'{e}odory functions, which may be nonmonotone, responsible for the coupling of our system, see the assumptions below.

In recent years, there have been several results investigating the gain of regularity of solutions of Maxwell--Schr\"odinger systems. As it turned out, even in rough regimes where the datum is not regular, the coupling in these systems allows the existence of solutions which somehow are more regular than expected, the so--called Regularizing Effect. By now, it is known that this is caused by the very interesting nature of the coupling between the equations of systems of Maxwell--Schr\"odinger type. Without the intention of being complete, the interested reader is invited to see \cite{ABO, boc1, boo1, boo2, boo3, dur, ASM} plus the references therein, and many others.

 The starting point for this analysis is due to L. Boccardo in \cite{boc1}, where this phenomenon was first investigated. Later, in a series of works in collaboration with L. Orsina, other aspects of regularity gains were explored, for instance see \cite{boo1,boo11}, or more recently \cite{boo3}, and the references therein. Summarizing, the authors prove existence/ regularizing properties of/for solutions for a class of local Maxwell--Schr\"{o}dinger systems.
 In order to illustrate the ideas, we quote that in \cite{boo1}, they  address the following problem
\begin{align}\label{P1.1}
\begin{cases}
&-\mbox{div}(M(x)\nabla u) + A \varphi |u|^{r-2}u=f, \ \ \  \  \ \mbox{in} \ \ \ \Omega; \\
&-\mbox{div}(M(x)\nabla \varphi) = |u|^r,\  \ \ \ \ \ \ \ \ \ \ \mbox{in} \ \ \ \ \Omega; \\
&u=\varphi=0 \ \ \ \mbox{on} \  \ \ \partial\Omega,
\end{cases}
\end{align}
where, from now on, $M(x)$  is a uniformly elliptic symmetric bounded matrix, i.e.,  $M(x)\xi \cdot\xi \geqslant \alpha|\xi|^{2}  \mbox{ and }  |M(x)|\leqslant \beta \ \ \mbox{a.e in} \ \Omega \ \mbox{and} \ \forall \ \xi \in \mathbb{R}^{N}$, for $0<\alpha\leqslant \beta$, $A$ is a positive constant, and $f\in L^m(\Omega)$. 
In short, among other results, the authors prove that even when $m<(2^*)^\prime$, if $r>2^*$ and $r^\prime \leqslant m$, there exist $u,\varphi \in W^{1,2}_0(\Omega)$, weak solution of \eqref{P1.1}. This is interesting since $f$ lies below the dual threshold $(2^*)^\prime$, $|u|^r$ is a priori low summable, and also because the weight $A\varphi$ is not assumed to be regular or strictly positive. Moreover, when $1<r<2^*$, if $\max\{1,\frac{Nr}{N+2r}\}< m< (2^*)^\prime $ there exists $(u,\varphi) \in W^{1,m^*}_0(\Omega) \times W^{1,q}_0(\Omega) $, solution of \eqref{P1.1} for 
\begin{equation*}q=\begin{cases} 2 \mbox{ if } \frac{2Nr}{N+2+4r}\leqslant m< (2^*)^\prime\\ \frac{Nm}{Nr-2mr-m}
\mbox{ if } \max\{1,\frac{Nr}{N+2r}\} < m < \frac{2Nr}{N+2+4r}.
\end{cases}
\end{equation*}

Later on,  in \cite{ASM}, analogous results were obtained to the following version of \eqref{P1.1} 

\begin{align}\label{P1.2}
\begin{cases}
&-\mbox{div}(M(x)\nabla u) + g(x,u,v)=f, \ \ \  \  \ \mbox{in} \ \ \ \Omega; \\
&-\mbox{div}(M(x)\nabla v) = h(x,u,v),\  \ \ \ \ \ \ \ \ \ \ \mbox{in} \ \ \ \ \Omega; \\
&u=v=0 \ \ \ \mbox{on} \  \ \ \partial\Omega,
\end{cases}
\end{align}
supposing that $f\geqslant 0$ a.e. in $\Omega$, while including two Carathéodory nonlinearities, $g(x,s,t)$ and $h(x,s,t)$, satisfying: \begin{enumerate}
\item[(a)] there exist $c_1$, $c_2>0$ such that
\begin{align}
\label{P1o}\tag{$\mbox{H}^\prime_1$}
c_1|s|^{r-1}|t|^{\theta+1} \leqslant |g(x,s,t)| \leqslant c_2|s|^{r-1}|t|^{\theta+1};
\end{align}
\item[(b)] $g(x,s,t)$ is monotone in $s$, i.e., \begin{align}\label{P2o}\tag{$\mbox{H}_2^\prime$}
(g(x,s_1,t)-g(x,s_2,t))(s_1-s_2)\geqslant0 \ \ \forall \ s_1,s_2 \in\mathbb{R}, t \geqslant 0  \ \ \mbox{a.e.}\  x\ \mbox{in} \ \Omega;
\end{align}
\item[(c)] there exist $d_1$, $d_2>0$ such that
\begin{align}
\label{P3}\tag{$\mbox{H}^\prime_3$}
d_1|s|^{r}|t|^{\theta} \leqslant |h(x,s,t)| \leqslant d_2|s|^{r}|t|^{\theta};
\end{align}
\item[(d)] $h(.,.,.)$ is non-negative
\begin{align}\label{P4}\tag{$\mbox{H}^\prime_4$} 
h(x,s,t)\geqslant0 \ \ \forall \ s,t \in\mathbb{R}, \ \ \mbox{a.e.}\  x\ \mbox{in} \ \Omega.
\end{align}
\end{enumerate}
In the particular case  where $m\geqslant (r+\theta+1)'$,  $0<\theta<\min\big(\frac{4}{N-2},1\big)$, $2^*<r+\theta+1 $ and $(r+\theta+1)^{\prime}\leqslant m<(2^*)^{\prime}$,  recalling that $f\in L^m(\Omega)$, the authors prove the existence of $(u,v) \in W^{1,2}_0(\Omega)\times W^{1,2}_0(\Omega)$, a regularized weak solution of \eqref{P1.2}.

For further developments on the local case, see \cite{ABO}, \cite{boo11}, and \cite{dur}. In \cite{ABO}, in cooperation with D. Arcoya, by assuming a interesting growth condition on $f(x)$,  it is addressed a version of \eqref{P1.1} with the constant $A>0$ replaced by a $A(x)\in L^1(\Omega)$, $A(x)\geqslant 0$, and also with a modification of the diffusion term on the second equation.  In \cite{boo11}, the authors include sublinear and singular structures in \eqref{P1.1}, i.e., the nonlinearities become $A\varphi^\theta u^{r-1}$ and $\frac{u^r}{\varphi^{1-\theta}}$ where $\theta\in(0,1).$  Moreover, for an investigation of a quasilinear version of \eqref{P1.1}, see \cite{dur}.


Meanwhile, a nonlocal version of \eqref{P} was introduced by L. Boccardo and L. Orsina in \cite{boo2}, where the authors addressed 
\begin{align}\label{S1.2}
    \begin{cases}
&-\mbox{div}\Big(\big(a(x)+\|\nabla u\|^{2}_{L^{2}}\big)\nabla u\Big) +  \varphi |u|^{r-2}u=f\ \ \  \ \ \mbox{in} \ \ \ \Omega; \\
&-\mbox{div}(M(x)\nabla \varphi) = |u|^r \ \ \ \  \  \ \ \mbox{in} \ \ \ \ \ \Omega; \\
&u=v=0 \ \ \ \ \ \ \mbox{on} \ \ \ \ \ \partial\Omega,
\end{cases}
\end{align}
a Kirchhoff–Schr\"{o}dinger–Maxwell (KMS) system, where $a:\Omega \to \mathbb{R}$ is a measurable function such that $0 <\alpha\leqslant a(x)\leqslant \beta. \ \ x $ a.e in $\Omega$, see \eqref{P1.1} above. The authors conduct a through analysis of existence, nonexistence and regularity of solutions of \eqref{S1.2} for $1<m<\infty$ and $r>1$. In particular, regarding regularizing zones, among other results, the authors proved that, if $ (r+1)^{\prime} \leqslant m <(2^*)^\prime$ and $r\geqslant 2^{*} -1$, or if $(2^*)^\prime\leqslant m$ and $1<r<2^*-1$, then there exists $(u,\varphi) \in W^{1,2}_{0}(\Omega) \times W^{1,2}_0(\Omega)$, solution of \eqref{S1.2}, where $u\in L^q(\Omega)$ for $q =\max \{{m^{**}, \frac{m(2r+1)}{m+1}}\}$. In the case $m<(2^*)^\prime$, clearly $q=\frac{m(2r+1)}{m+1}$ ensuring a double regularizing effect: there exist energy solutions despite that $f\notin W^{-1,2}(\Omega)$ and the $u$ has improved integrability since $q>m^{**}$, see Definition  \ref{regularizing}.

 Further, let us add that as pointed out in \cite{boo2}, outside the energy zone, e.g., $1<m<\min\{(r+1)^\prime, (2^*)^\prime\}$, there might be interesting convergence phenomenon. Indeed, they consider $f\in L^{1}(\Omega)\setminus W^{-1,2}(\Omega)$, $1<r<2$ and $N=6$. In this case, there exists a sequence of approximate solutions $(u_n,\varphi_n)$ where $u_n,\varphi_n \in W^{1,2}_0(\Omega)\cap L^\infty(\Omega)$, and as it turned out, under these assumptions, they prove that $\|\nabla u_n\|^2_{L^2}\to +\infty$ and $u_n \rightharpoonup  0$ weakly in $W^{1,q}_0(\Omega)$, for $1<q<\frac{18}{11}$. Yet, by setting $v_n= \|\nabla u_n\|^2_{L^2}u_n$, they show that $v_n \rightharpoonup w$, the entropy solution of $-\Delta w = f$ and moreover, since $|u_n|^r \to 0$ strongly in $L^{1}(\Omega)$, then $\varphi_n \rightharpoonup 0$ weakly in $W^{1,\rho}_0(\Omega)$, $1<\rho<\frac{6}{5}$. As a consequence, the second equation degenerates and converges to $0=0$.
 
 In the present paper, we are interested in the analysis of regularity properties of solutions to a quasilinear version of the nonlocal Maxwell--Schrodinger system given by \eqref{P}. As a result, we have to deal with certain obstacles arising from the modifications that we introduced.
   If on one hand, we combine ideas from \cite{boo2,dur} and \cite{ASM}, at the same time, we have to circumvent that our operators are not self--adjunct, causing some challenges when translating duality arguments from Laplacian to $p$-Laplacian equations, and yet work with a class of zeroth-order of nonlinearities. In particular, this makes the approach based on approximation and passage to the limit slightly more complex. 

Our main purpose is to prove the existence of regularized distributional solutions to \eqref{P}. With this in mind and in order to avoid the aforementioned degeneration of the limit solutions appearing in KMS systems, c.f. \cite{boo2} and as described above, we restrict our analysis to energy zones for the datum $f\in L^m(\Omega)$, meaning that $m> (r+\theta+1)^\prime$, see our assumptions below.

\subsection*{Basic Assumptions}
Throughout the text we  assume the hypotheses below.

We consider $r>1$, $0<\theta < p-1$, $1<p<N$, $1<m<\frac{N}{P}$, $g:\Omega\times\mathbb{R}\times\mathbb{R}\to \mathbb{R}$ and $h:\Omega\times\mathbb{R}\times\mathbb{R}\to \mathbb{R}$, both Carath\'{e}odory, satisfying	

\begin{enumerate}
		\item[(i)]\ there exist constants $c_1>0$, $c_2>0$,  such that
		\begin{align}
			\label{H1}\tag{$\mbox{H}_1$}
		 c_1|s|^{r}|t|^{\theta+1}&\leqslant g(x,s,t)s 	 
		\\
		\label{H2}\tag{$\mbox{H}_2$}
			|g(x,s,t)| &\leqslant c_2|s|^{r-1}|t|^{\theta+1} 
   		\end{align}
for all $ s, t \in\mathbb{R}$ a.e. $x$ in  $\Omega$;
  
  \item[(ii)] there exist constants $d_1>0$, $d_2>0$  such that
		\begin{align}
			\label{H3}\tag{$\mbox{H}_3$}
			 d_1|s|^{r}|t|^{\theta+1} &\leqslant h(x,s,t)t
		\\
  \label{H4}\tag{$\mbox{H}_4$}
  |h(x,s,t)| &\leqslant d_2|s|^{r}|t|^{\theta}
	\end{align}		
    for all  $s, t \in\mathbb{R}$ a.e.  $x$ in  $\Omega$.
		
	\end{enumerate}
 	Let us briefly remark that the prototypes for $g$ and $h$ are given by $g(s,t)=s|s|^{r-2}|t|^{\theta+1}$ and $h(s,t)=t|s|^{r}|t|^{\theta-1}$. Moreover, we stress that \eqref{H1}-\eqref{H4} are slightly weaker than the hypotheses assumed in \cite{ASM}, see \eqref{P1o}-\eqref{P4}. Grossly speaking,  we now allow nonmonotone nonlinearities which may change sign. For example, we can consider \[g(x,s,t)=V_1(x)s|s|^{r-2}|t|^{\theta+1}(\mbox{cos}(\pi s)+\pi) \mbox{ and } h(x,s,t)=V_2(x)t|s|^{r}|t|^{\theta-1}(\mbox{sin}(\pi t)+\pi),\] where $V_i(x)\geqslant e_i>0$ a.e. in $\Omega$, $i=1,2$, which satisfy \eqref{H1}-\eqref{H4} but not \eqref{P1o}-\eqref{P4}.  
\subsection*{Main Results}
Our main results concern existence of solutions, and also the investigation of regularizing properties of system \eqref{P}. 
In order to pursue our main objectives, for the sake of clarity, we have decided to include our definition of solutions for \eqref{P}.

\begin{definition}
Let $1<p<N$. We say that $(u,v)$ in $W^{1,p}_{0}(\Omega)\times W^{1,p}_{0}(\Omega)$, is a solution for problem \eqref{P} if and only if 
\begin{align}\displaystyle\tag{$P_{F}$}\label{P_F}
\begin{cases}
&\displaystyle \mathcal{A}(u,v)\int_{\Omega}|\nabla u|^{p-2}\nabla u \cdot \nabla \varphi  + \int_{\Omega} g(x,u,v)\varphi = \int_{\Omega} f\varphi \ \ \ \ \forall\ \varphi \in C^\infty_c(\Omega)\\
& \displaystyle \mathcal{A}(u,v)\int_{\Omega}|\nabla v|^{p-2}\nabla v\cdot\nabla\psi = \int_{\Omega} h(x,u,v)\psi \ \ \ \ \ \forall\ \psi \in C^\infty_c(\Omega),
\end{cases}
\end{align}
$g(.,u,v)$ and $h(.,u,v) \in L^1_{loc}(\Omega)$, where $\mathcal{A}(u,v)= \|\nabla u\|^{p}_{L^{p}}+\|\nabla v\|^{p}_{L^{p}}$.
\end{definition}

Remark that the latter definition of solutions is intermediate. Indeed, on one hand they are not regular enough to be considered as weak solutions, for instance, we cannot take test functions in $W^{1,p}_0(\Omega)$. On the other hand, they are better than standard distributional or even ``very weak" solutions since $(u,v)\in W^{1,p}_0(\Omega)\times W^{1,p}_0(\Omega)$.


Throughout the text, we are going to discuss whether or not our pair $(u,v)$ gains regularity contrasted to the standard regularity theory for the associated decoupled equations.  For this, to simplify the statements of our results, following \cite{ASM}, we introduce the notion below of {\bf Sobolev and Lebesgue Regularized} solutions.

For this, from now on, given $1<p<N$ and $1\leqslant \tau < \frac{N}{p}$ we set 
\begin{equation}
    \label{1755}
\tau_p^*=\tau^* (p-1)= \dfrac{N\tau(p-1)}{N-\tau} \mbox{ and } \tau^{**}_p=(\tau_p^*)^*= \dfrac{N\tau(p-1)}{N-\tau p}. \end{equation}

Based on the regularity results obtained for $p$-Laplacian-like equations in \cite{bbggpv, bog2}, we consider the following concept of regularized solutions:
\begin{definition}\label{regularizing}
Consider $w\in W^{1,1}_0(\Omega)$ a distributional solution of
\[-\Delta_p w = F\in L^\tau(\Omega).\]
\begin{itemize}
    \item[(a)] We say that $w$ is Lebesgue regularized if $w\in L^q(\Omega)$ for $q>\max\{\tau^{**}_p,1\};$
    \item[(b)] We say that $w$ is Sobolev regularized if $\nabla w \in W^{1,q}_0(\Omega)$ where $q>\max\{\tau_p^*,1\}.$
\end{itemize}
\end{definition}
This definition formalizes the idea that, somehow sometimes, solutions could be a little better than what is guaranteed by the classic regularity machinery. 
Remark that there are no regularized zones where solutions are as regular as predicted by standard regularity theory.

We are now in position to introduce our first result.

 \begin{theorem} \label{theorem1}
		Let $f\in L^{m}(\Omega),$ where $f\geqslant0$ a.e in $\Omega,$ $ \min\{(r+\theta+1)^{\prime}, (p^*)^{\prime}\}<m<\frac{N}{p}$ and $0<\theta< \min\{p-1, \frac{p^2}{N-p}\}.$ Then there exists a  solution $(u,v)$ for \eqref{P}, with $u\in W^{1,p}_{0}(\Omega) \cap L^{r+\theta+1}(\Omega)$, $u\geqslant0$ a.e in $\Omega$ and $v\in W^{1,p}_{0}(\Omega)$, $v\geqslant0$ a.e. in $\Omega$.    
	\end{theorem}
Our approach to prove Theorem \ref{theorem1} is based on the analysis of approximate problems allowing us to gather delicate integral estimates. This is achieved by means of the careful choice of tailored test functions. Of course, in the end, we pass to the limit.

This theorem connects results from \cite{boo2}, \cite{dur} and \cite{ASM}, since it reaches a nonlocal, $p$-Laplacian-like Maxwell-Schr\"{o}dinger system, for a broader class of nonlinearities. This combination, brings new difficulties to a problem which is already complex. If on one hand, we consider $g$ and $h$ satisfying weaker conditions than previous results, on the other hand, the coupling between these nonlinearities and the nonlocal terms makes the problem even more challenging, especially the passage to the limit of the approximate problems, which is very tricky.   Regarding the restriction on $m>m_0=(r+\theta+1)^\prime$ mentioned before, this is considered to avoid known degeneration issues for limit solutions of Kirchhoff-Maxwell-Schr\"{o}dinger systems, as explained above and described in \cite{boo2}.

In the particular case of the solutions provided by Theorem \ref{theorem1}, the regularizing zones are established in the next result.
\begin{corollary}\label{ER}
Consider $(u,v)$ a solution for \eqref{P}, given by Theorem \eqref{theorem1}. 
\begin{itemize}
    \item[(I)] If $r+\theta>p^{*} -1$ and $(r+\theta+1)^{\prime} <m< (p^{*})^{\prime}$, then $u$ is Sobolev regularized.
    \item[(II)] If $r+\theta>p^{*} - 1$ and $(p^{*})^{\prime}\leqslant m < \frac{N(r+\theta+1)}{N(p-1)+p(r+\theta+1)} $, then $u$ is Lebesgue regularized. 
    \item[(III)] If $r+\theta> p^{*} -1$, then $v$ is Sobolev regularized. 
\end{itemize}
    
\end{corollary}
Observe that we do not consider $2<r+\theta+1\leqslant p^*$ in the regularizing zones given by Corollary \ref{ER}. This is due to the fact that  we take $m>\min\{(r+\theta+1)^{\prime},(p^*)^\prime\}$. Then,  if $r+\theta+1\leqslant p^*$ which is equivalent to  $m\geqslant(p^*)^\prime$ we get $m^*_p\geqslant p$ and $m^{**}_p\geqslant p^*$, so that out solutions are outside the regularizing zones.

Another consequence of  Theorem \ref{theorem1}  is that, under its hypotheses, the solutions $(u,v)$ are not trivial or semitrivial. In general terms, the principle is that nonsmooth solutions cannot be trivial.

\begin{corollary}\label{C1.2}
    Consider $(u,v)$ solution given by Theorem \ref{theorem1}. Then $u$ is nontrivial, i.e., $u\neq 0$ at least in a set of positive Lebesgue measure. Moreover, if $m<(p^*)^\prime$ then, $v$ is also nontrivial, ie. $v\neq 0$ at least in a set of positive Lebesgue measure. 
\end{corollary} 
 The interesting fact is that this nontrivial behavior of solutions comes directly from the non smooth character of the data where we obtain existence of solutions.
 In particular, if $r+\theta+1\leqslant m <(p^*)^\prime$, our solutions are purely vector, see \cite{liliane}.

\section{Preliminaries} 

The first result addresses a preliminary version for an approximate problem of \eqref{P} based on adequate truncations of $g$, $h$. Indeed, from now on, given $\eta>0$, by recalling that  $T_\eta(s)=\min\{\eta,\max\{-\eta,s\}\}$,  we denote
\begin{align*}
		g_\eta(x,t,s)=T_\eta(g(x,t,s)) \ \ \ \mbox{and} \ \ \
			h_\eta(x,t,s)=T_\eta(h(x,t,s)).	
	\end{align*}
	It is clear that $g_\eta$ and $h_\eta$  are bounded and also $g_\eta(x,s,t)s\geqslant 0$ a.e. in $\Omega$. We are in a position to provide the following proposition, based on the Browder--Minty Theorem, see \cite{lionsbook}, Thm. 2.7, p. 180.

	\begin{proposition}\label{prop1}
 Given $F\in L^t(\Omega)$, $\eta,\tau>0$, if $t\geqslant (p^*)^{\prime}$,  there exists a couple $(u,v)\in W^{1,p}_0(\Omega) \times W^{1,p}_0(\Omega)$ weak solution of
	\begin{align}\label{FistE}
	    & \mathcal{A}_\tau(u,v)\int_{\Omega} |\nabla u|^{p-2}\nabla u \cdot \nabla \varphi + \int_{\Omega} g_{\eta}(x,u,v)\varphi = \int_{\Omega}F \varphi   \\
     & \label{SecondE} \mathcal{A}_\tau(u,v)\int_{\Omega} |\nabla v|^{p-2}\nabla v \cdot \nabla \psi = \int_{\Omega} h_{\eta}(x,u,v)\psi+\int_\Omega\dfrac{\psi}{\tau}\ \ \ \forall \ \varphi, \psi \in W^{1,p}_0(\Omega),
	\end{align}
 where $\mathcal{A}_\tau(u,v)\ = \dfrac{1}{\tau}+ \|\nabla u\|^{p}_{L^{p}}+ \|\nabla v\|^{p}_{L^{p}}$. 
  
   Moreover, if $t>\dfrac{N}{p}$ and $0<\theta<\dfrac{p^2}{N-p}$,  then $u, v \in L^\infty(\Omega)$ and 
			\begin{align}
			    \label{1749}
   \|u\|_{L^\infty} &\leqslant \dfrac{C\|F\|^{\frac{1}{p-1}}_{L^t}}{(\|\nabla u\|^p_{L^p}+\|\nabla v\|^p_{L^p}+\min{\{1,\tau^{-1}\}})^{\frac{1}{p-1}}}\leqslant C\|F\|_{L^t}^{\frac{1}{p-1}}\max\{1,\tau^{\frac{1}{p-1}}\},\\
\nonumber   \|v\|_{L^\infty} &\leqslant \dfrac{C\tau^{\frac{1}{p-1}}\|F\|_{L^t}^{\frac{r}{(p-1)^2}}\|\nabla v\|_{L^p}^{\frac{\theta}{p-1}}}{(\|\nabla u\|^p_{L^p}+\|\nabla v\|^p_{L^p}+\min{\{1,\tau^{-1}\}})^{\frac{r}{(p-1)^2}}}+1\\
\label{1932}
&\leqslant C\Big(\|F\|_{L^t}^{\frac{r(2p-1)+\theta(p-1)}{(p-1)^2(2p-1)}}\max\{1,\tau^{\frac{r+p-1}{(p-1)^2}}\}+1\Big),
            \end{align}
	where $C=C(c_2,d_2, p,r,t, N,\theta,\Omega)>0$ and is uniform with respect to $\eta>0$.
    \end{proposition}
	
	\begin{proof}
		Consider $ T: W^{1,p}_0(\Omega)\times W^{1,p}_0(\Omega)\longrightarrow W^{-1,p^\prime}(\Omega) \times W^{-1,p^\prime}(\Omega)$ determined by
	\begin{align*}
	    \big(T(u,v),(\varphi,\psi)\big)&=  \mathcal{A}_\tau(u,v)\int_{\Omega} \bigg( |\nabla u|^{p-2} \nabla u\cdot \nabla \varphi + |\nabla v|^{p-2}\nabla v \cdot \nabla \psi\bigg) \\
     & + \int_{\Omega} g_{\eta}(x,u,v)\varphi-\int_{\Omega}F \varphi- \int_{\Omega} h_{\eta}(x,u,v) \psi-\int_\Omega \dfrac{\psi}{\tau}.
	\end{align*}
As a initial step, remark that  $T$ is well--defined. Our purpose is to verify that $T$ satisfies the conditions of the Browder--Minty Surjectivity theorem. Indeed, by \eqref{H1}, since $g_\tau(x,u,v)u\geqslant 0$,  discarding the positive terms and employing standard Sobolev embeddings we get
\begin{align*}
	\big(T(u,v),(u,v)\big)&\geqslant \frac{1}{\tau} \big(\|u\|^p_{W^{1,p}_0}+\|v\|^p_{W^{1,p}_0}\big)- C \|F
    \|_{L^t} \|u\|_{W^{1,p}_0} - C (\eta+\dfrac{1}{\tau}) \|v\|_{W^{1,p}_0}  \\
 & \geqslant C(\eta,\tau) \big(\|u
 \|^p_{W^{1,p}_0}+\|v\|^p_{W^{1,p}_0}\big)\Big(1 - \frac{\|F\|_{L^t}\|u\|_{W^{1,p}_0} + 
 \|v\|_{W^{1,p}_0}}{\|u\|^p_{W^{1,p}_0}+\|v\|^{p}_{W^{1,p}_0}}\Big),
\end{align*}  
guaranteeing that $T$ is coercive. 

Further, consider $\varphi, \psi \in W^{1,p}_0(\Omega)$, such that $\|\varphi\|_{W^{1,p}_0}+\|\psi\|_{W^{1,p}_0}\leq 1$,  and $u,v\in W^{1,p}_0(\Omega)$, where $\|u\|_{W^{1,p}_0}+\|v\|_{W^{1,p}_0}\leqslant \rho$, for a given $\rho>0$. Thus, by recalling that $|g_\eta(x,u,v)|\leqslant \eta $ and $|h_\eta(x,u,v)|\leqslant \eta $, a.e. in $\Omega$, by standard applications of H\"{o}lder's and Sobolev's inequalities,  we arrive at 
\begin{align*}
	\big|(T(u,v),(\varphi,\psi))\big|
	&\leqslant  C( \tau) \Big(1+\mathcal{A}(u,v)\Big)(\|u\|^{p-1}_{W^{1,p}_0}\|\varphi\|_{W^{1,p}_0}+ \|v\|^{p-1}_{W^{1,p}_0}\|\psi\|^p_{W^{1,p}_0}\Big) \\
 &+ C(N,p,\eta, \Omega)\big(\|F\|_{L^{t}}\|\varphi\|_{W^{1,p}_0}+ \|\varphi\|_{W^{1,p}_0}+\|\psi\|_{W^{1,p}_0}\big) \\
 &\leqslant  C(N,p,\eta, \tau, \rho, \Omega, \|F\|_{L^t}).
\end{align*}
Then $T$ is bounded whenever $\|u\|_{W^{1,p}_0}+\|v\|_{W^{1,p}_0}$ is bounded.


Now, in order to complete the verification that $T$ is pseudo-monotone, suppose that $u_k \rightharpoonup u$, $v_k \rightharpoonup v$ weakly in $W^{1,p}_0(\Omega)$, and that \[\limsup\limits_{k\to\infty} \big(T(u_k,v_k),(u_k-u,v_k-v)\big)\leqslant 0.\] 
It is straightforward to check that 
\[ \lim_{k\to\infty}\bigg(\int_{\Omega} g_{\eta}(x,u_k,v_k)(u_k-u)-\int_{\Omega}F (u_k-u)- \int_{\Omega} h_{\eta}(x,u_k,v_k) (v_k-v)-\int_\Omega \dfrac{v_k-v}{\tau}\bigg)=0.\]
Then, since $A_\tau(u_k,v_k)> \frac{1}{\tau}$, for all $k \in \mathbb{N}$, by recalling the choices of $T$ and $\mathcal{A}_k$, as an immediate consequence, we have 
\begin{equation*}
\limsup\limits_{k\to\infty} \Bigg(\int_{\Omega}|\nabla u_k|^{p-2}\nabla u_k\cdot\nabla( u_k -u)+\int_{\Omega}|\nabla v_k|^{p-2}\nabla v_k \cdot \nabla(v_k -v)\Bigg) \leqslant 0.
\end{equation*}

Thus, by combining the latter inequality with Lemma \ref{mintytrick}, we obtain that
\[\lim_{k\to \infty}\|u_k-u\|_{W^{1,p}_0}=\lim_{k\to \infty}\|v_k-v\|_{W^{1,p}_0} =0,\]
as claimed.

         In this way, we clearly have \begin{align*}
			\liminf_{k\to\infty}{\big(T(u_k,v_k),(u_k-\varphi,v_k-\psi\big)}\geqslant \big(T(u,v), (u-\varphi,v - \psi)\big), \ \forall \varphi,\psi \in W^{1,p}_0(\Omega).
		\end{align*}
	
		
        Hence, $T$ is pseudomonotone and by the Minty--Browder Theorem, there exists $(u,v) \in W^{1,p}_0(\Omega) \times W^{1,p}_0(\Omega)$ satisfying \eqref{FistE}-\eqref{SecondE}. Moreover, since $F\in L^t(\Omega)$ where $t\geqslant (p^*)^\prime$, by taking $\varphi=u$ in \eqref{FistE} and $\psi=v$ in \eqref{SecondE}, so that by \eqref{H1}, \eqref{H4}, and straightforward manipulations we have

\begin{align*}
    \mathcal{A}_\tau(u,v)\int_{\Omega} |\nabla u|^p + c_1 \int_{\Omega} |u|^r|v|^{\theta+1} \leqslant C\| F\|_{L^t} \|u\|_{W^{1,p}_0} \mbox{ and } \mathcal{A}_\tau(u,v) \int_{\Omega} |\nabla v|^p \leqslant d_2 \int_{\Omega} |u|^r|v|^{\theta+1}
\end{align*}
so that
\begin{equation}
    \label{1857}
\| u\|_{W^{1,p}_0}+\| v\|_{W^{1,p}_0}\leqslant C\bigg(\tau\|F\|_{L^t}\bigg)^{\frac{1}{p-1}} \mbox{ and } \| u\|_{W^{1,p}_0}+\| v\|_{W^{1,p}_0}\leqslant C\|F\|^{\frac{1}{2p-1}}_{L^t}
\end{equation}

As a final step, we handle the $L^{\infty}$ estimates for $u$ and $v$.  For this, we apply Lemma \ref{linfinity} twice. 
First, observe that if we consider
\[H(x,s)=\frac{g_\eta(u,v)}{\mathcal{A}_\tau(u,v)},\]
then by \eqref{H1}, we have $H(x,s)s\geqslant 0$ a.e. in $\Omega$, so that $u\in W^{1,p}_0(\Omega)$ is a weak solution of 

 \begin{equation*}
 \begin{cases}
 -\Delta_p u+H(x,u)=F_\tau(x) \mbox{ in } \Omega\\
 u=0 \mbox{ on } \partial \Omega,
 \end{cases}
 \end{equation*}
where \[F_\tau(x)=\dfrac{F(x)}{\mathcal{A}_\tau(u,v)}\in L^t(\Omega), t> \frac{N}{p}\]
so that by Lemma \ref{linfinity} (a), there holds that $u\in L^\infty(\Omega)$. Moreover,
by noticing that
\[\|F_\tau\|_{L^t}\leqslant \dfrac{\|F\|_{L^t}}{\|\nabla u\|_{L^p}^p+\|\nabla v\|_{L^p}^p+\tau^{-1}}  \]
we end up with
\[\|u\|_{L^\infty}\leqslant \dfrac{C\|F\|^{\frac{1}{p-1}}_{L^t}}{\big(\|\nabla u\|_{L^p}^p+\|\nabla v\|_{L^p}^p+\tau^{-1}\big)^{\frac{1}{p-1}}}.\]

Finally, now considering
\[H(x,s)=\frac{h_\eta(u,s)+\tau^{-1}}{\mathcal{A}_\tau(u,v)},\]
we clearly have that $v\in W^{1,p}_0(\Omega)$ is a weak solution to
 \begin{equation*}
 \begin{cases}
 -\Delta_p v=H(x,v) \mbox{ in } \Omega\\
 v=0 \mbox{ on } \partial \Omega,
 \end{cases}
 \end{equation*}
while $u\in L^\infty(\Omega)$ together with \eqref{H4} guarantees that $H(x,s)\leqslant d_2\tau\|u\|_{L^\infty}^r|s|^{\theta}+1$ a.e. in $\Omega$, which is uniform with respect to $\eta$. 

Then, for $\alpha=1$ and $\beta= d_2\tau\|u\|_{L^\infty}^r $, by combining the latter inequalities with Lemma \ref{linfinity} (b), \eqref{1857}, we have $v\in L^\infty(\Omega)$ and
\begin{align*}\|v\|_{L^\infty}\leqslant C\bigg(\tau^{\frac{1}{p-1}}\|u\|_{L^\infty}^{\frac{r}{p-1}}\|v\|^{\frac{\theta}{p-1}}_{W^{1,p}_0}+1\bigg)&\leqslant \dfrac{C\bigg(\tau^{\frac{1}{p-1}}\|F\|^{\frac{r}{(p-1)^2}}_{L^t}\|v\|^{\frac{\theta}{p-1}}_{W^{1,p}_0}\bigg)}{\big(\|\nabla u\|_{L^p}^p+\|\nabla v\|_{L^p}^p+\tau^{-1}\big)^{\frac{r}{(p-1)^2}}}+1\\
&\leqslant \dfrac{C\tau^{\frac{1}{p-1}}\|F\|^{\frac{r}{(p-1)^2}+{\frac{\theta}{(p-1)(2p-1)}}}_{L^t}}{\big(\|\nabla u\|_{L^p}^p+\|\nabla v\|_{L^p}^p+\tau^{-1}\big)^{\frac{r}{(p-1)^2}}}+1,
\end{align*}
so that the result follows.

	\end{proof}
	The next step is to consider a suitable approximate version \eqref{P}, carefully chosen for proving the existence of adequate solutions to our original problem. 
    
	\begin{proposition}\label{cor1}
		Let $f\in L^{m}(\Omega)$ with $m\geqslant1$ and $0<\theta<\frac{p^2}{N-p}$. Then there exist  $u_k$ and $v_k$, both in  $W^{1,p}_0(\Omega) \cap L^\infty(\Omega)$, weak solution  to 
    \end{proposition}
       \vspace{-0.6cm}
        \begin{align}
        \label{aprox}
		\begin{cases}
				 \ -\mbox{div} \Big(\mathcal{A}_k(u_k,v_k)|\nabla u_k|^{p-2}\nabla u_k\Big)+g(x,u_k,v_k)= f_k \mbox{ in } \Omega,\\ \ -\mbox{div}\Big(\mathcal{A}_k(u_k,v_k)|\nabla v_k|^{p-2}\nabla v_k\Big)=h(x,u_k,v_k)+ \frac{1}{k} \mbox{ in } \Omega,\\
                u_k=v_k=0 \mbox{ on } \partial \Omega,
			\end{cases}
		\end{align}
\textit{where $f_k = T_k(f)$, $\mathcal{A}_k(u_k,v_k)=\|\nabla u_k\|_{L^p}^p+\|\nabla v_k \|^p_{L^p}+\frac{1}{k}$.}

\textit{
     In addition, if $f\geqslant0$ a.e. in $\Omega$, then $u_k\geqslant 0$ and  $v_k > 0$ a.e in $\Omega$. }
	
	\begin{proof}
As a first step, we prove existence. For this, let us start by pointing out that $f_k\in L^\infty(\Omega)$ and that $|f_k|\leqslant \min\{|f|,k\}$ a.e. in $\Omega$, for all $k>0$, since $T_k(s)=\max (-k,\min(s,k))$. 

    In view of Proposition \ref{prop1}, given $k>0$, let us fix $F=f_k$, $\tau=k$ and take $\eta>0$ satisfying $\eta> C\big(k^q+1\big)$, for
 \[q=\frac{2r}{p-1}+(\theta+1)\frac{2r(2p-1)+(\theta+2p-1)(p-1)}{(p-1)^2(2p-1)},\] where $C$ is the maximum between the constants in \eqref{1749}, \eqref{1932}, $c_2$ and $d_2$. Hence,  by Proposition \ref{prop1}, there exist $u_k$, $v_k\in W^{1,p}_0(\Omega)\cap L^{\infty}(\Omega) $ solution of \eqref{FistE} and \eqref{SecondE}.
		Moreover, it is clear that $g_\tau(.,u_k,v_k)=g(.,u_k,v_k)$ and $h_\tau(.,u_k,v_k) = h(.,u_k,v_k)$ a.e. in $\Omega$. Indeed, by recalling \eqref{H2}, \eqref{H4}, \eqref{1749} and \eqref{1932} we get
		$|g(x,u_k,v_k)| \leqslant c_2 |u_k|^{r-1}|v_k|^{\theta+1}\leqslant C(k^q+1)<\eta$ and $|h(x,u_k,v_k)| \leqslant d_2 |u_k|^{r}|v_k|^{\theta}\leqslant C(k^q+1)<\eta$ a.e. in $\Omega$, proving the claim. In this way, $(u_k,v_k)$ is a weak solution of \eqref{aprox}.
		
		Next, suppose in addition that $f\geqslant 0$ a.e. in $\Omega$, so that  $f_k \geqslant 0$ a.e. in $\Omega$. Further, notice that by \eqref{H1} and \eqref{H3}, both $g(x,s,t)s\geqslant 0$ and $h(x,s,t)t\geqslant 0$ a.e. in $\Omega$. Then,
        it is imediate from the Weak Maximum Principle that $u_k\geqslant 0$ and $v_k\geqslant 0$ a.e. in $\Omega.$ 

        Finally, for $\psi\in W^{1,p}_0(\Omega)$ such that $\psi \geqslant 0$ a.e. in $\Omega$, observe that
        \begin{equation}
        \label{1543}
        \mathcal{A}_k(u_k,v_k)\int_\Omega |\nabla v_k|^{p-2}\nabla v_k\cdot \psi=\int_\Omega h(x,u_k,v_k)\psi+ \frac{\psi}{k}>0
        \end{equation}
        and in particular
        \[ \int_\Omega |\nabla v_k|^{p-2}\nabla v_k\cdot \psi\geqslant0.\]
       Therefore by the Strong Maximum Principle for the $p$-Laplacian if $v_k=0$ in a set of positive measure we would have $v_k=0$ a.e. in $\Omega.$ However, by \eqref{1543} we conclude that $v_k>0$ a.e. in $\Omega$.
        
	\end{proof}
	
	\section{Estimates and Convergence}
	We are ready to address our crucial a priori estimates, providing the basis for our main results. Here, the main difficulty lies in handling the mixed integral estimates, which are delicate.

\begin{proposition}\label{lem402}
		Let $f \in L^m(\Omega)$, $f\geqslant 0$ a.e. in $\Omega$, where $m>  \min\{ (r+\theta+1)^\prime,(p^*)^{\prime}\}$,  $r>1$ and $0<\theta< \min\{p-1, \frac{p^2}{N-p}\}.$ Then
		\begin{align}\label{esti1}
			\|u_k\|_{L^{r+\theta+1}}\leqslant C\|f\|^{\sigma}_{L^m} \ \ \mbox{and} \ \ \|u_k\|^{2p}_{W^{1,p}_0} + \|v_k\|^{2p}_{W^{1,p}_0} +  	\int_\Omega u_k^r v_k^{\theta+1} \leqslant C \|f\|^{\sigma + 1}_{L^m},
		\end{align}
where  \begin{align*}
			\sigma = \begin{cases}
				\frac{2p}{2p-1} \ \ \mbox{if} \  m^{\prime}\geqslant r+\theta+1,\\ 
               \frac{1}{r+\theta} \ \ \mbox{if} \ \ m^\prime <r+\theta+1.
			\end{cases}
		\end{align*}
Moreover, there exist $u\in W^{1,p}_0(\Omega) \cap L^{r+\theta+1}(\Omega)$ and $v\in W^{1,p}_0(\Omega)$, such that $u,v\geqslant 0$ a.e. in $\Omega$  and
\begin{equation}
\label{1745}
\begin{cases}
u_k \rightharpoonup u \mbox{ weakly in }  W^{1,p}_0(\Omega)\cap L^{r+\theta+1}(\Omega),\\
v_k \rightharpoonup v \mbox{ weakly in }  W^{1,p}_0(\Omega),\\
u_k \to u \mbox{ a.e. in  }  \Omega \mbox{ and strongly in } L^q(\Omega),  1\leqslant q <\max\{p^*, r+\theta+1\},
\\
v_k \to v \mbox{ a.e. in  }  \Omega \mbox{ and strongly in } L^q(\Omega), 1\leqslant q <p^*,\
\end{cases}
\end{equation}
	up to subsequences relabeled the same.
    
    \end{proposition}
	
    \begin{proof}
		Let us begin by taking $u_k$ as a test function in the first equation of \eqref{aprox}, so that we have 
		\begin{align}
        \nonumber
			\mathcal{A}_k(u_k,v_k) \int_\Omega |\nabla u_k|^p + \int_\Omega g(x,u_k,v_k)u_k \leqslant \int_\Omega f_k u_k,
		\end{align}
 recalling that we denote 
\[\mathcal{A}_k(u_k,v_k)=\frac{1}{k}+\|\nabla u_k\|^p_{L^{p}}+\|\nabla v_k\|^p_{L^{p}} \mbox{ and }\mathcal{A}(u_k,v_k)=\|\nabla u_k\|^p_{L^{p}}+\|\nabla v_k\|^p_{L^{p}}.\]
    In this fashion,  by $\eqref{H1}$, since $|f_k|\leqslant |f|$ a.e. in $\Omega$,  we end up with 
	\begin{align}\label{est.03.3}
	     \mathcal{A}_k(u_k,v_k)\int_\Omega |\nabla u_k|^p + c_1 \int_\Omega u_k^r v_k^{\theta+1}  \leqslant C\|f\|_{L^m}\|u_k\|_{L^{m^{\prime}}} \ 
	\end{align}
		where we have used the H\"older inequality and that  $f\in L^m(\Omega)$.
		
Now, by taking $v_k$ in the second equation of \eqref{aprox},  by \eqref{H4} it is clear that 
		\begin{align}
\nonumber 			 \mathcal{A}_k(u_k,v_k)\int_\Omega |\nabla v_k|^p \leqslant \int_\Omega h(x,u_k,v_k)v_k 
   \leqslant  d_2 \int_\Omega u_k^r v_k^{\theta+1}.
		\end{align}
By plugging the last estimate in  \eqref{est.03.3} we arrive at
		\begin{align} \label{est3.5}
			\mathcal{A}_k(u_k,v_k)\Bigg(\int_\Omega |\nabla u_k|^p + \int_\Omega |\nabla v_k|^p  \Bigg)  \leqslant C \|f\|_{L^m}\|u_k\|_{L^{m^{\prime}}}, 
		\end{align}
for $C=C(c_1,d_2,r,\theta,\Omega)>0$.        

At this point, clearly, it is enough to control $\|u_k\|_{L^{m^{\prime}}}$, what will be achieved by means of a tailored test function. With this purpose in mind, recalling that $u_k\geqslant 0$, and $v_k>0$ a.e. in $\Omega$, we denote $$\psi= u_k^{\theta+1}(v_k+\varepsilon)^{-\theta}$$ and notice that $\psi\in W^{1,p}_0(\Omega)$, by the locally Lipschitz Chain Rule. In this way, by taking $\psi$ as a test function in the second equation of \eqref{aprox}, we arrive at
		\begin{align*}
			\int_\Omega h(x,u_k,v_k)u_k^{\theta+1}(v_k+\varepsilon)^{-\theta} &= (\theta+1) \mathcal{A}_k(u_k,v_k) \int_\Omega |\nabla v_k|^{p-2}\nabla v_k \cdot\nabla u_k u_k^\theta (v_k+\varepsilon)^{-\theta}\\ 
			& -\theta \mathcal{A}_k(u_k,v_k)  \int_\Omega |\nabla v_k|^{p-2}\nabla v_k \cdot\nabla v_k \ u_k^{\theta+1}(v_k+\varepsilon)^{-(\theta+1)}.
		\end{align*}
		Then, as a direct consequence of \eqref{H3} there follows that
		\begin{align}\label{K4.6}
			\nonumber\int_\Omega u_k ^{r+\theta+1} v_k^{\theta}(v_k+\varepsilon)^{-\theta} &+ \theta  \mathcal{A}_k(u_k,v_k) \int_\Omega |\nabla v_k|^p u_k^{\theta+1} (v_k+\varepsilon)^{-(\theta+1)}\\
			& \leqslant (\theta+1) \mathcal{A}_k(u_k,v_k) \int_\Omega |\nabla v_k|^{p-1} |\nabla u_k| u^{\theta}(v_k+\varepsilon)^{-\theta}.
		\end{align}
		Now, let us handle the last term on the right-hand side of \eqref{K4.6}, which is delicate. Indeed, by the Young inequality for $p$ and $p^\prime$, given $\xi>0$
		\begin{align}\label{K4.7}
			 \int_\Omega |\nabla v_k|^{p-1} |\nabla u_k| u_k^{\theta}(v_k+\varepsilon)^{-\theta}\leqslant C(\xi)  \int_{\Omega} |\nabla u_k|^p 
			+ \xi \int_{\Omega} |\nabla v_k|^p u_k^{\frac{p\theta}{p-1}}(v_k+\varepsilon)^{-\frac{p\theta}{p-1}}
		\end{align}
		Further, by denoting $E_1=\{x\in\Omega;\ u_k< v_k+\varepsilon\}$ and $E_2=\{x\in\Omega;\ u_k \geqslant v_k+\varepsilon\}$, it is obvious  that $\Omega=E_1\cup E_2$, and hence 
		\begin{align*}
			 \int_{\Omega} |\nabla v_k|^p u_k^{\frac{p\theta}{p-1}}(v_k+\varepsilon)^{-\frac{p\theta}{p-1}}  
			\leqslant \int_{E_1} |\nabla v_k|^p 
             +\int_{E_2} 
			|\nabla v_k|^p u_k^{\frac{p\theta}{p-1}}(v_k+\varepsilon)^{-\frac{p\theta}{p-1}}
            \\
            \leqslant \int_{E_1} |\nabla v_k|^p 
             +\int_{E_2} 
			|\nabla v_k|^p u_k^{\theta+1}(v_k+\varepsilon)^{-(\theta+1)},
		\end{align*}
		where for the last term, we have used that $\frac{p\theta}{p-1}<\theta+1$, since $\theta < p-1$, and that $\frac{u_k}{v_k+\varepsilon}>1$ a.e. in $E_2$. Thus, by combining the previous estimate with \eqref{K4.6} and \eqref{K4.7}, it is straightforward that 
		
		\begin{align*}
			\int_\Omega u_k ^{r+\theta+1} v_k^{\theta}(v_k+\varepsilon)^{-\theta} &+ \theta \mathcal{A}_k(u_k,v_k) \int_\Omega |\nabla v_k|^p u_k^{\theta+1} (v_k+\varepsilon)^{-(\theta+1)}\\
			&\leqslant C(\xi)  (\theta+1)\mathcal{A}_k(u_k,v_k)\int_{\Omega} |\nabla u_k|^p \\
			& + \xi (\theta+1)\mathcal{A}_k(u_k,v_k) \Bigg(\int_\Omega |\nabla v_k|^p+ \int_\Omega 
			|\nabla v_k|^pu_k^{\theta+1}(v_k+\varepsilon)^{-(\theta+1)}\Bigg).
		\end{align*}
		Now, by choosing $ \xi =  \frac{\theta}{\theta+1}$ and $C = \max \{ C(\xi)  (\theta+1), \theta\}$, and canceling out some terms,         by \eqref{est3.5} we end up with
		\begin{align*}
			\int_\Omega u_k ^{r+\theta+1} v_k^{\theta}(v_k+\varepsilon)^{-\theta} 
			&\leqslant C\mathcal{A}_k(u_k,v_k)\bigg(\int_{\Omega} |\nabla u_k|^p+\int_{\Omega} |\nabla v_k|^p \bigg)
            \\
            &\leqslant C \|f\|_{L^m}\|u_k\|_{L^{m^{\prime}}}.
		\end{align*}
Thus, by the Fatou Lemma, recalling that $v_k>0$ a.e in $\Omega$, if $\varepsilon \to 0$ we arrive at
		\begin{align}\nonumber
			\int_\Omega u_k ^{r+\theta+1}\leqslant \liminf_{\varepsilon \to0} \int_\Omega u_k ^{r+\theta+1} v_k^{\theta}(v_k+\varepsilon)^{-\theta}  &\leqslant C A(u_k,v_k) \Big(\int_\Omega |\nabla u_k|^p + \int_\Omega |\nabla v_k|^p\Big)\\
			 &\label{1534}\leqslant C \|f\|_{L^m}\|u_k\|_{L^{m^{\prime}}}.
		\end{align}
        
        Now, notice that if $m^\prime \geqslant r+\theta+1$, since $m>\min\{ (r+\theta+1)^\prime, (p^*)^\prime \}$ we have $p^*> m^\prime$, so that, by \eqref{est3.5} we have $\|u_k\|_{W^{1,p}_0}\leqslant C \|f\|_{L^m}^{\frac{1}{2p-1}}$. Then, by \eqref{1534}, we get 
         \begin{align*}
         \|u_k\|_{L^{r+\theta+1}}\leqslant C\|f\|_{L^m}^{\sigma}, \mbox{ where }
			\sigma = \begin{cases}
				\frac{2p}{2p-1} \ \ \mbox{if} \ \ m^\prime \geqslant r+\theta+1,\\ 
                \frac{1}{r+\theta} \ \ \mbox{if} \ \ m^\prime < r+\theta+1.
			\end{cases}
		\end{align*}
		Finally,  by plugging the last estimate in \eqref{est.03.3} and \eqref{est3.5}, recalling the choices of $\mathcal{A}$ and $\mathcal{A}_k$, we obtain
		\begin{align*}
 \ \|u_k\|^{2p}_{W^p_0} + \|v_k\|^{2p}_{W^p_0} +  	\int_\Omega u_k^r v_k^{\theta+1} \leqslant C \|f\|^{\sigma+1}_{L^m},
		\end{align*}
		where we have discarded the positive terms multiplying $k^{-1}$.

The proof of \eqref{1745} is standard. For instance, by recalling that $u_k\geqslant 0$ and $v_k>0$ a.e in $\Omega$, it is enough to combine \eqref{esti1}, the Relich--Kondrachov theorem, Thm. 3.1 and Thm. 3.3 in \cite{bocc}.
        
	\end{proof}
In the next result, we address the passage to the limit for the mixed nonlinearities and related terms. To handle the most delicate items, we ground on Lemma \ref{nonlinearconvergences} in the Appendix.
 
\begin{proposition}\label{convergence}
		\eqref{H1}-\eqref{H4}. Consider $0<\theta<\min \{p-1,\frac{p^2}{N-p}\}$, $1<p<N$, and $m>\min\{(p^*)^\prime, (r+\theta+1)^\prime\}$, and $u_k,v_k$ solution of \eqref{aprox} given by Proposition \ref{cor1}.
	Then, up to subsequences relabeled the same.	
        \begin{itemize}
			\item[(a)] $g(.,u_k,v_k)u_k  \rightarrow g(.,u,v)u \ \ {\rm in} \ L^1(\Omega)$, 
		\item[(b)] $g(.,u_k,v_k)  \rightarrow g(x,u,v) \ \ {\rm in} \ L^1(\Omega)$, 	
            
			\item[(c)] 
            $ h(.,u_k,v_k)v_k \rightarrow h(.,u,v)v \ \ {\rm in} \ L^1(\Omega)$.
            \item[(d)]  $ h(.,u_k,v_k) \rightarrow h(.,u,v) \ \ {\rm in} \ L^1(\Omega)$.
            \end{itemize}
				\end{proposition}
	\begin{proof} We are going to invoke Lemma \ref{nonlinearconvergences} several times, see p. \pageref{nonlinearconvergences} below.
\begin{itemize}
    \item Proof of (a);
\end{itemize}
By Proposition \ref{lem402}, we already know that $u_k\to u$, $v_k\to v$ a.e. in $\Omega$, $\{u_k\}_\mathbb{N}$ is bounded in $L^{r+\theta+1}(\Omega)\cap L^{p^*}(\Omega)$, and $\{v_k\}_\mathbb{N}$ is bounded in $ L^{p^*}(\Omega)$. In the notation of Lemma \ref{nonlinearconvergences}, we set $q_1=\max\{r+\theta +1, p^*\}$, $q_2=p^*$, $\tilde{u}_k=u_k$ and $\tilde{v}_k=v_k$, in a way that (i) and (ii) are already satisfied.
Moreover, let $H(x,s,t)=g(x,s,t)s$ so that by \ref{H2} $H(x,s,t)\leqslant c_2 |s|^r|t|^{\theta+1}$ a.e. in $\Omega$. By these choices, 
\begin{equation*}
\int_{\{u_k\leqslant n_0\}}|H(x,u_k,v_k)| \leqslant c_2n_0^r \int_{\{u_k\leqslant n_0\}}v_k^{\theta+1}, \forall k\in \mathbb{N},
\end{equation*}
and then, we fix $\omega_2\equiv 1$, $s_2=\theta+1$ and $t_2=+\infty$, and $(iv)$ is satisfied, see p. \pageref{nonlinearconvergences}.
Further, observe that by taking $\varphi=\frac{1}{\lambda}T_\lambda(G_n(u_k))u_k$ as a test function in the first equation of \eqref{aprox}, by noticing that
\[\mathcal{A}_k(u_k,v_k)\int_\Omega |\nabla u_k|^{p-2}\nabla u_k \cdot \nabla \varphi \geqslant 0, \] we end up with
\[\int_\Omega g(x,u_k,v_k) \frac{T_\lambda(G_n(u_k)}{\lambda} \leqslant \int_{\{u_k>n\}}f u_k,\]
since $T_\lambda(G_n(u_k))=0 $ where $u_k\leqslant n$. Moreover, by recalling that
\[\frac{1}{\lambda}T_\lambda(G_n(u_k))=1 \mbox{ a.e. in }\{u_k>n+\lambda\}\mbox{ and } g(x,u_k,v_k)u_k\dfrac{T_\lambda(G_n(u_k))}{\lambda} \geqslant 0 \mbox{ a.e. in } \Omega,\]
it is clear that
\[\int_{\{u_k>n+\lambda\}} g(x,u_k,v_k)u_k\leqslant \int_{\Omega} g(x,u_k,v_k)u_k \dfrac{T_\lambda(G_n(u_k))}{\lambda}.\]
In this fashion, by combining the latter estimates, there follows
\[\int_{\{u_k>n+\lambda\}} g(x,u_k,v_k)u_k\leqslant \int_{\{u_k>n\}}f u_k, \forall k\in \mathbb{N} \mbox{ and } \forall\lambda>0.\]
Thus, by taking $\lambda\to 0$, the Fatou Lemma implies
\begin{equation*}
\int_{\{u_k>n\}}H(x,u_k,v_k)=\int_{\{u_k>n\}}g(x,u_k,v_k)u_k \leqslant \int_{\{u_k>n\}}f u_k, \forall k\in \mathbb{N},
\end{equation*}

and once again, under the notation of Lemma \ref{nonlinearconvergences}, we set $\omega_1=f$, $t_1=m$ and $s_1=1$ so that $(iii)$ is satisfied.

At this point, it is enough to recall that, since $q_1=\max\{r+\theta+1,p^*\}$, $s_1=1$, $t_1=m$, $q_2=p^*$, $s_2=\theta+1$ and $t_2=+\infty$, one clearly has that 
\[q_1>s_1, q_2>s_2 \mbox{ and } t_2>\dfrac{q_2}{q_2-s_2}.\]

In addition, remark that since $m>\min\{(p^*)^\prime, (r+\theta+1)^\prime\}$ then
\[t_1=m>\max\bigg\{\dfrac{r+\theta+1}{r+\theta}, \dfrac{p^*}{p^*-1}\bigg\}=\dfrac{q_1}{q_1-s_1}.\]
Hence, by Lemma \ref{nonlinearconvergences} there follows that
\[g(.,u_k,v_k)u_k=H(.,u_k,v_k)\to H(.,u,v)=g(.,u,v)u \mbox{ strongly in } L^1(\Omega),\]
proving (a).



\begin{itemize}
    \item Proof of (b);
\end{itemize}

Remark that , given $E\subset \Omega$, we have
\begin{align*}
    \int_E| g(x,u_k,v_k)| &\leqslant \int_{E\cap\{u_k\leqslant 1\}}  | g(x,u_k,v_k)|+\int_{E\cap\{u_k\geqslant1\}}  |g(x,u_k,v_k)u_k|\\
    &\leqslant c_2\int_E v_k^{\theta+1} +\int_E   g(x,u_k,v_k)u_k,
\end{align*}
where we have employed hypotheses \eqref{H1} and \eqref{H2}. Moreover, recall that by item $(a)$ $\{g(.,u_k,v_k)u_k\} _\mathbb{N}$ is uniformly integrable and $\{v_k^{\theta+1}\}_\mathbb{N}$ is uniformly integrable since $\theta+1<p^{*}$.
In this way, $\{ g(.,u_k,v_k)\}_{\mathbb{N}}$ is also uniformly integrable and hence by the Vitali convergence theorem we prove (b).
\begin{itemize}
    \item Proof of (c);
\end{itemize}
It is clear that $h(x,u_k,v_k)v_k \to h(x,u,v)v$ a.e. in $\Omega.$
Next, observe that, since $g(.,u_k,v_k)u_k\to g(.,u,v)u$ in $L^1(\Omega)$ and by \eqref{H1} $c_1 u_k^r v_k^{\theta+1} \leqslant g(x,u_k,v_k)u_k$ a.e. in $\Omega$, we already have that
\begin{equation*}
u_k^rv_k^{\theta+1} \to u^{r}v_k^{\theta+1} \mbox{ in } L^1(\Omega).
\end{equation*}
However, by \eqref{H4}, there holds $h(x,u_k,v_k)v_k\leqslant d_2 u_k^rv_k^{\theta+1}$ a.e. in $\Omega$, so that by the Lebesgue Dominated Convergence Theorem, $h(.,u_k,v_k)v_k \to h(.,u,v)v$ in $L^1(\Omega)$. 
\begin{itemize}
    \item Proof of (d);
\end{itemize}
It is analogous to the proof of (b).

	\end{proof}

	\section{Existence and Regularizing Effects}
	We are now in position to prove our main results.
    \subsection{Existence of Solutions - Proof of Theorem \ref{theorem1}}
	\begin{proof}
		Due to Proposition \ref{lem402}, we already know that there exist  $u$, $v$ in $W^{1,p}_0(\Omega)$ such that, up to subsequences relabeled the same, 
		\begin{align*}
			&u_k,v_k\rightharpoonup u, v  \ \mbox{in} \ W^{1,p}_0(\Omega), \ u_k \to u \ \mbox{in} \ L^{q_1}(\Omega), \  \ v_k \to v \ \mbox{in} \ L^{q_2}(\Omega) \ \mbox{and a.e. in} \ \Omega;
		\end{align*}
and
		\begin{align}\label{es12.2}
     \|u\|_{L^{r+\theta+1}}\leqslant C\|f\|^{\sigma} \ \ \mbox{and} \ \ \|u\|^{2p}_{W^{1,p}_0} + \|v\|^{2p}_{W^{1,p}_0} +  	\int_\Omega u^r v^{\theta+1} \leqslant C \|f\|^{\sigma+1}_{L^m},
		\end{align}
where $u\geqslant 0$, $v\geqslant 0$ a.e. in $\Omega$,  $q_1< \max\{p^*,r+\theta+1\}$,  $q_2< p^*$, and $\sigma$ is given in Proposition \ref{lem402}. 
Further, by Proposition \ref{nonlinearconvergences}	we have  
		\begin{align*}
			g(.,u_k,v_k) \to g(.,u,v) \mbox{ and } h(.,u_k,v_k) \to h(.,u,v)\ \mbox{in} \ L^1(\Omega).
		\end{align*}
		
        Thus,  in order to pass to the limit in system \eqref{aprox} and prove that $(u,v)$ is a solution to \eqref{P} it remains to establish the convergence of the nonlocal terms on both equations. Of course, this is a mere consequence of a stronger convergence of $u_k$ and $v_k$. As a matter of fact, we will show that, up to subsequences, $u_k\to u$ and $v_k \to v$ strongly in $W^{1,p}_0(\Omega).$
          For this purpose, let us set $\Gamma = \lim_{k} A(u_k,v_k)$ and suppose that without loss of generality $\Gamma >0$.
        
        At this point, let us consider  $u_k - T_n(u)$ as a test functions in \eqref{aprox}. Then
        \[\bigg(\frac{1}{k}+1\bigg)\mathcal{A}(u_k,v_k)\int_{\Omega} |\nabla u_k|^{p-2}\nabla u_k \cdot\nabla(u_k - T_n(u))+ \int_{\Omega} g (x,u_k,v_k)(u_k - T_n(u)) =\int_{\Omega} f_k (u_k - T_n(u))\]
        By subtracting on both sides
        \[ \mathcal{A}(u_k,v_k)\int_{\Omega}|\nabla T_n(u)|^{p-2} \nabla T_n(u) \cdot\nabla(u_k - T_n(u)),\]
        we end up with
\begin{align}
\label{est4.4}
		\nonumber \mathcal{A}(u_k,v_k) &\int_{\Omega}\bigg(|\nabla u_k|^{p-2}\nabla u_k - |\nabla T_n(u)|^{p-2} \nabla T_n(u) \bigg)\cdot\nabla(u_k-T_n(u))\\ &\nonumber + \int_{\Omega} g(x, u_k, v_k)(u_k - T_n(u)) = \int_{\Omega} f_k (u_k - T_n(u)) 
			\\ &   -\bigg(\frac{1}{k}+\mathcal{A}(u_k,v_k)\bigg)\int_{\Omega}|\nabla T_n(u)|^{p-2} \nabla T_n(u) \cdot\nabla(u_k - T_n(u)).  
		\end{align}  

       By taking $\limsup_k$ in \eqref{est4.4} we arrive at 
        \begin{align*}
            \Gamma\limsup_k &\int_{\Omega}\bigg(|\nabla u_k|^{p-2}\nabla u_k - |\nabla T_n(u)|^{p-2} \nabla T_n(u) \bigg)\nabla(u_k-T_n(u))\\ &\nonumber + \int_{\Omega} g(x, u, v)(u - T_n(u)) = \int_{\Omega} f G_n(u) -\Gamma\int_{\Omega}|\nabla T_n(u)|^{p-2} \nabla T_n(u) \cdot G_n(u)
        \end{align*}
Since $f_k \to f$ in $L^m(\Omega)$, $u_k \rightharpoonup u$ in $L^1(\Omega)$, $u_k \rightharpoonup u$ in $W^{1,p}(\Omega)$, $|\nabla T_n(u) |^{p-2} \nabla T_n(u) \in L^{p^{\prime}}(\Omega)$ $g(x,u_k,v_k)u_k \to g(x,u,v)u$ in $L^1(\Omega)$ and $g(x,u_k,v_k) \to g(x,u,v)$ in $L^1(\Omega)$ where we used that
\begin{align*}
        m >  \min\{ (r+\theta+1)^\prime,(p^*)^{\prime}\} \iff \ m^{\prime} < \max \{r+\theta+1, p^{*}\}.
\end{align*}

Also note that, $\nabla T_n(u)$ and $\nabla G_n(u)$ are orthogonal to each other, thus by discarding the positive term we have
 \begin{align*}
            \limsup_k &\int_{\Omega}\bigg(|\nabla u_k|^{p-2}\nabla u_k - |\nabla T_n(u)|^{p-2} \nabla T_n(u) \bigg)\nabla(u_k-T_n(u)) \leqslant \frac{1}{\Gamma}\int_{\Omega} f G_n(u). 
        \end{align*}
        However, by Lemma \ref{monotonicityplaplacian}
        \begin{align*}
            \| \nabla u_k -\nabla T_n(u)\|^p_{L^p} \leqslant \int_{\Omega}\bigg(|\nabla u_k|^{p-2}\nabla u_k - |\nabla T_n(u)|^{p-2} \nabla T_n(u) \bigg)^\alpha\cdot \big(1+\|\nabla u_k\|^p_{L^p} + \|\nabla T_n(u)\|^p_{L^p}\big)^{\beta}.
        \end{align*}
        And then by combining the last estimates with \eqref{esti1} we get

        \begin{align*}
            \limsup_k \|\nabla u_k - \nabla T_n(u)\|^p_{L^p} \leqslant \frac{C}{\Gamma^\alpha}\bigg(\int_{\Omega} f G_n(u)\bigg)^\alpha.
        \end{align*}
        By recalling that $\| \nabla u_k - \nabla u \|^p_{L^p} \leqslant C \bigg(\| \nabla u_k - \nabla T_n(u)\|^p_{L^p} + \|\nabla G_n (u)\|^p_{L^p}\bigg)$, 
       we arrive at 
        \begin{align}\nonumber
\limsup_k \| \nabla u_k - \nabla u \|^p_{L^p} &\leqslant C\limsup_k \bigg(\| \nabla u_k - \nabla T_n(u)\|^p_{L^p} + \|\nabla G_n (u)\|^p_{L^p}\bigg) \\
\label{0814}
&\leqslant  \frac{C}{\Gamma^\alpha}\bigg(\int_{\Omega} f G_n(u)\bigg)^\alpha + \| \nabla G_n(u)\|^p_{L^p}.
        \end{align}
Moreover, since $\mbox{meas}(\{u>n\}) \to 0$ if $n\to + \infty,$ it is straightforward that
\[\lim_n \int_{\Omega} |\nabla G_n(u)|^p = \lim_n \int_{\{u>n\}} |\nabla u|^p = 0 \mbox{ and }\lim_n \int_{\Omega} f G_n(u) = \lim_n \int_{\{u>n\}} f u = 0. \]
 Thus, given $\varepsilon>0$, there exists $n_0\in\mathbb{N}$ such that 
 \[\frac{C}{\Gamma^\alpha}\bigg(\int_{\Omega} f G_n(u)\bigg)^\alpha + \| \nabla G_n(u)\|^p_{L^p}<\varepsilon, \forall n\geqslant n_0.\]
Plugging the latter inequality in \eqref{0814}, we have
\begin{align*}
    \limsup_k \| \nabla u_k -\nabla u\|^p_{L^p} <\varepsilon, \forall \varepsilon>0,
\end{align*}
and hence by letting $\epsilon \to 0$, we prove $\nabla u_k \to \nabla u$ in $L^p(\Omega)$. This is obvious when $\Gamma =0$. 

Finally, we prove that $\nabla v_k \to \nabla v$ a.e in $\Omega$. The argument is analogous to the last one.
Indeed, consider $v_k - T_n(v)$ a test function in \eqref{aprox} 
so that, after subtracting on both sides
\[\mathcal{A}(u_k,v_k)\int_{\Omega} |\nabla T_n(v)|^{p-2}\nabla T_n(v)\cdot \nabla (v_k - T_n(v))\]
by taking the $\limsup_k$ on both sides, once again, thanks to Lemma \eqref{monotonicityplaplacian}, 
by noticing that 
\begin{align*}
				\lim_{k\to \infty}\int_{\Omega} |\nabla T_n(v)|^{p-2} \nabla T_n(v) \cdot \nabla (v_k - T_n(v)) = 0, 
			\end{align*}
since $v_k\rightharpoonup v$ in $W^{1,p}_0(\Omega)$,  it is straightforward that
\begin{align*}    \limsup_{k\to \infty}\int_{\Omega} |\nabla (v_k-T_n(v))|^p \leqslant \frac{C}{\Gamma^\alpha}\limsup_{k\to \infty}\bigg( \int_\Omega  \big(h(x,u_k,v_k)+\frac{1}{k}\big)(u_k-T_n(v)\bigg)^\alpha.
\end{align*}
Moreover, recall that $v_k\to v$ in $L^1(\Omega)$, and by Proposition \ref{convergence} $h(.,u_k,v_k)\to h(.,u,v)$  and $h(.,u_k,v_k)v_k\to h(.,u,v)v$, in $L^1(\Omega)$. Thence, since $T_n(v)\in L^\infty(\Omega)$, we have $h(,u_k,v_k)(u_k-T_n(v))\to h(,u_k,v_k)G_n(v)$ in $L^1(\Omega)$, if $k\to +\infty$. In this way, the latter inequality implies
\begin{align*} 
\limsup_{k\to \infty}\int_{\Omega} |\nabla (v_k-T_n(v))|^p \leqslant \frac{C}{\Gamma^\alpha}\bigg( \int_\Omega  h(x,u,v) G_n(v)\bigg)^\alpha, \forall n\in \mathbb{N} \mbox{, where } C \mbox{ is uniform in }n.
\end{align*}
Now, we claim that $h(.,u,v)G_n(v)\to 0 $ in $L^1(\Omega)$. If this were true, then by arguing as in \eqref{0814} and the subsequent arguments, we would prove that $\nabla v_k \to \nabla v$ in $L^p(\Omega)$, finishing our proof. In this way, recalling  \ref{H3}, since $G_n(v)\leqslant v$ a.e. in $\Omega$, one has
\[\int_\Omega h(x,u_k,v_k)G_n(v)\leqslant d_2\int_{\{v>n\}}u^{r+1}v^{\theta+1}.\]
However, by \eqref{es12.2} $\int_\Omega u^r v^{\theta+1}<+\infty$ and then, since $\mbox{meas}(\{v>n\})\to 0$ and $n\to +\infty$, by the Monotone Convergence Theorem
\[h(.,u,v)G_n(v)\to 0 \mbox{ in } L^1(\Omega)\]
and the result follows.

Therefore we can pass to the limit on the approximate problem  \eqref{aprox} and get that $(u,v)$ is a solution to problem \eqref{P}.		
       
	\end{proof}
\subsection{Regularizing Effects - Proof of Corollary \ref{ER}}	
First, observe that as $r+\theta>p^*-1$, clearly, $\min\{ (r+\theta+1)^\prime, (p^*)^\prime\}=(r+\theta+1)^\prime$.
\begin{itemize}
    \item Proof of (I);
\end{itemize}
 It is immediate that $p>m_p^{*} \iff m<(p^{*})^{\prime}$ and $m<(p^{*})^{\prime} \iff m_p^{**}< p^{*}$, see \eqref{1755} p. \pageref{1755}. Thence, since $(r+\theta+1)^{\prime}<m< (p^{*})^{\prime}$, and by Theorem \ref{theorem1},  $u\in W^{1,p}_0(\Omega)\cap L^{r+\theta+1}(\Omega)$, in the current case $u$ is Sobolev regularized. Notice that, in particular, since $m^{**}_p < p^* =\min\{ p^*, r+\theta+1\}$, 
clearly, $u$ is also Lebesgue regularized.

\begin{itemize}
    \item Proof of (II);
\end{itemize}
Now, $(p^*)^\prime \leqslant m < \frac{N(r+\theta+1)}{N(p-1)+p(r+\theta+1)}$. Then, by Theorem \ref{theorem1}, we already know that $u\in W^{1,p}_0(\Omega)\cap L^{r+\theta+1}(\Omega)$ and $u$ cannot be Sobolev regularized. 
However,  it is clear that
\[m<\frac{N(r+\theta+1)}{N(p-1)+p(r+\theta+1)} \iff m_p^{**}<r+\theta+1.\]
Thus, is $u$ is Lebesgue regularized.
\begin{itemize}
    \item Proof of (III);
\end{itemize}
According to Theorem \eqref{theorem1} we know that $u\in L^{r+\theta+1}(\Omega)$ and $v \in L^{p^{*}}(\Omega)$, in such a way that $u^{r} \in L^{\frac{r+\theta+1}{r}}(\Omega)$ and $v^{\theta} \in L^{\frac{p^{*}}{\theta}}(\Omega)$. Thus, by hypothesis \eqref{H4} and the interpolation inequality we have
\begin{align*}
    \|h(x,u,v)\|_{L^t} \leqslant \Big(d_2^{t}\int_{\Omega}| |u|^{r}|v|^{\theta}|^{t}\Big)^{\frac{1}{t}} \leqslant d_2 \|u^{r}\|_{L^{r+\theta+1}}\|v^{\theta}\|_{L^{\frac{p^{*}}{\theta}}}< \infty,
\end{align*}
where $t = \frac{p^{*}(r+\theta+1)}{p^{*}r+\theta(r+\theta+1)}$.
Therefore, as $r+\theta> p^{*} -1 \iff t < (p^{*})^{\prime}$ it follows that $v$ is Sobolev regularized.
\subsection{Proof of Corolary \ref{C1.2}}

\begin{proof}
The idea is to use the non-smoothness of the data to discard trivial solutions. 
First we prove that $u\neq 0$. Indeed, by $u=0$, then by \eqref{H2} we would have $g(.,u,v)\equiv 0$. Thus, since the pair $(u,v)$ satisfies Definition \eqref{P_F}, then we would have $\int_\Omega f \varphi =0$ for all $\phi\in C^\infty_c(\Omega)$ and in particular $f\in L^\infty(\Omega)$ what contradicts the fact that $f\in L^m(\Omega)$ and $m<\frac{N}{p}.$ In this way, $u$ is nontrivial.

Next, we prove that $v$ is nontrivial. For this, we suppose that $m<(p^*)^\prime$ and consider the distribution $A$ defined by 
		\[(A(u),\varphi) = \int_\Omega\Big(\|\nabla u\|^p_{L^p}+\|\nabla v\|^p_{L^p}\Big) |\nabla u|^{p-2} \nabla u \cdot \nabla \varphi + \int_\Omega g(x,u,v) \varphi \ \ \forall \varphi \in C^{\infty}_c(\Omega),\]
the left-hand side of the first equation of \eqref{P_F}.
        
        Indeed, if $v=0$ a.e. in $\Omega$, by \eqref{H2} one has $g(x,u,v)\equiv0 \in L^{\infty}(\Omega)$. Then,  we are allowed to extend $A$ continuously to all $W^{1,p}_0(\Omega),$ since it becomes \[(A(u),\varphi) = \int_\Omega\Big(\|\nabla u\|^p_{L^p}\Big) |\nabla u|^{p-2} \nabla u \cdot \nabla \varphi, \forall \varphi \in W^{1,p}_0(\Omega).\] In particular,  since $u$ solves $\eqref{P_F}$
		so that $f = A(u) \in W^{-1,p}(\Omega)$ which  produces a contradiction, seeing  $f\in L^m(\Omega)$ with $m\leqslant (p^*)^{\prime}$. Thus $\Lambda \neq 0$, furthermore as $\{\Lambda _k\}^{\infty}_k$ is a positive sequence, there follows that $\Lambda >0$.	
\end{proof}

\section{Appendix}
In order to concentrate on the essential aspects and not to burden the proofs with too many technical details, we have decided to include an appendix to discuss the next result. For the sake of clarity, we have kept their proofs and formulation separate from the rest of the text.

\subsection*{Monotonicity and Regularity for the $p$-Laplacian}

To begin with, we address some standard results regarding the monotonicity of the $p$-Laplacian. For the convenience of the reader, the details are included.
\begin{lemma}
\label{monotonicityplaplacian}
Given $1\leqslant p<+\infty$, for all  $\tilde{u}$ end $\tilde{v}$ in $\in W^{1,p}_{0}(\Omega)$, there holds that
\[\|\tilde{u}-\tilde{v}\|^p_{W^{1,p}_0}\leq C \Bigg(\int_\Omega \big(|\nabla \tilde{u}|^{p-2}\nabla \tilde{u}-|\nabla \tilde{v}|^{p-2}\nabla \tilde{v}\big)\cdot (\nabla \tilde{u}-\nabla \tilde{v})\Bigg)^\alpha \bigg(1+\|\nabla \tilde{u}\|^p_{L^{p}}+\|\nabla \tilde{v}\|^p_{L^{p}}\bigg)^\beta,  \]
 where $C(p)=\max\{ 2^{p-2}, \big(\frac{2^{\frac{2-p}{2}}}{p-1}\big)^\frac{p}{2}\}$, $\alpha=\min\big\{\frac{p}{2},1\big\}$\ and $\beta=\max\big\{\frac{2}{2-p},0\big\}$.
\end{lemma}
\begin{proof}
Recall that given vectors $A$ and $B\in \mathbb{R}^N$
\begin{equation}
\label{ineqplapla}
\begin{cases}
\Big( |A|^{p-2}A-|B|^{p-2}B\Big)\cdot\Big( A-B\Big) &\geqslant \dfrac{|A-B|^p}{2^{p-2}}, \mbox{ if } 2\leqslant p\leqslant 2,\\

\Big( |A|^{p-2}A-|B|^{p-2}B\Big)\cdot\Big( A-B\Big) &\geqslant \dfrac{(p-1)|A-B|^2}{\big(1+|A|^2+|B|^2\big)^{\frac{2-p}{2}}}, \mbox{ if } 1<p\leqslant 2.
\end{cases}
\end{equation}
In the first case,  for $p\geqslant 2$  \eqref{ineqplapla} clearly implies
\[\int_\Omega |\nabla \tilde{u}-\nabla \tilde{v}|^p\leq 2^{p-2} \int_\Omega \big(|\nabla \tilde{u}|^{p-2}\nabla \tilde{u}-|\nabla \tilde{v}|^{p-2}\nabla \tilde{v}\big)\cdot (\nabla\tilde{u}-\nabla \tilde{v})  \]
and by noticing that in this case $\alpha=1$ and $\beta=0$, our result follows, 

For the second case, $1< p <2$, by \eqref{ineqplapla} we have 
\begin{equation} \nonumber\int_\Omega \dfrac{|\nabla \tilde{u}-\nabla \tilde{v}|^2}{\big( 1+|\nabla \tilde{u}|^2+|\nabla \tilde{v}|^2\big)^{\frac{2-p}{2}}}\leq \frac{1}{p-1} \int_\Omega \big(|\nabla \tilde{u}|^{p-2}\nabla \tilde{u}-|\nabla \tilde{v}|^{p-2}\nabla \tilde{v}\big)\cdot (\nabla \tilde{u}-\nabla \tilde{v})  \end{equation}

Thus,  by using that $1<p<2$, it is enough to notice that by the H\"older inequality for $\frac{2}{p}$ and $\frac{2}{2-p}$, it is straightforward that
\begin{align*}
\int_\Omega |\nabla \tilde{u}-\nabla \tilde{v}|^p &\leq \Bigg(
\int_\Omega \dfrac{|\nabla \tilde{u}-\nabla \tilde{v}|^2}{\big( 1+|\nabla \tilde{u}|^2+|\nabla \tilde{v}|^2\big)^{\frac{2-p}{2}}}\Bigg)^{\frac{p}{2}}
\Bigg(
\int_\Omega \big( 1+|\nabla \tilde{u}|^2+|\nabla \tilde{v}|^2\big)^{\frac{p}{2}}\Bigg)^{\frac{2-p}{2}}\\
&\leq 2^{\frac{(2-p)p}{4}}\Bigg(
\int_\Omega \dfrac{|\nabla \tilde{u}-\nabla \tilde{v}|^2}{\big( 1+|\nabla \tilde{u}|^2+|\nabla \tilde{v}|^2\big)^{\frac{2-p}{2}}}\Bigg)^{\frac{p}{2}}
\Bigg(
\int_\Omega \big( 1+|\nabla \tilde{u}|^p+|\nabla \tilde{v}|^p\big)\Bigg)^{\frac{2-p}{2}}.
\end{align*}
The result follows by combining the last two inequalities.

\end{proof}
Next, for the sake of completeness, we state and prove a known convergence property of the $p$-Laplace operator.
\begin{lemma}\label{mintytrick} Consider $\{\tilde{u}_n\}_{\mathbb{N}}\subset W^{1,p}_0(\Omega)$ and $ \tilde{u}\in W^{1.p}_0(\Omega)$, where $\tilde{u}_n\rightharpoonup \tilde{u},$ weakly in $W^{1,p}_0(\Omega)$. 

\begin{itemize}
    \item[(a)] If we suppose
\end{itemize}    
    \begin{equation*}  \limsup_n \int_\Omega |\nabla \tilde{u}_n|^{p-2}\nabla \tilde{u}_n \cdot (\nabla \tilde{u}_n-\tilde{u})\leqslant 0.
\end{equation*}
Then $\tilde{u}_n \rightarrow \tilde{u}$ in $W^{1,p}_0(\Omega)$ strongly.
\begin{itemize}
\item[(b)] Moreover, also consider $ \{\tilde{v}_n\}_{\mathbb{N}}\subset W^{1,p}_0(\Omega)$ and $ \tilde{v}\in W^{1.p}_0(\Omega)$, where $ \tilde{v}_n \rightharpoonup \tilde{v}$ in $W^{1,p}$. If we suppose that
\end{itemize}    
    \begin{equation} \label{1448} \limsup_n \bigg(\int_\Omega |\nabla \tilde{u}_n|^{p-2}\nabla \tilde{u}_n \cdot (\nabla \tilde{u}_n-\tilde{u})+\int_\Omega |\nabla \tilde{v}_n|^{p-2}\nabla \tilde{v}_n \cdot (\nabla \tilde{v}_n-\nabla \tilde{v})\bigg)\leqslant 0.
\end{equation}
Then $\tilde{u}_n \rightarrow \tilde{u}$ and $\tilde{v}_n \rightarrow \tilde{v}$ strongly in $W^{1,p}_0(\Omega)$.
\end{lemma}
\begin{proof}
Notice that (a) is a direct consequence of (b). For instance, if we admit $\tilde{v}_n=0$ for all $n$, then (b) clearly implies (a). 

Thus, notice that from \eqref{1448} one clearly has that
\begin{align*}
\nonumber \limsup\limits_{n\to\infty} &\Bigg(\int_{\Omega}\Big(|\nabla \tilde{u}_n|^{p-2}\nabla \tilde{u}_n-|\nabla \tilde{u}|^{p-2}\nabla \tilde{u}\Big)\cdot \big(\nabla \tilde{u}_n -\nabla \tilde{u}\big)\\
&+\int_{\Omega}\Big(|\nabla \tilde{v}_n|^{p-2}\nabla \tilde{v}_n-|\nabla \tilde{v}|^{p-2}\nabla \tilde{v} \Big)\cdot \big(\nabla \tilde{v}_n -\nabla \tilde{v}\big)\Bigg) \\
&\leqslant \limsup_{k\to\infty}\Bigg(\int_{\Omega}|\nabla \tilde{u}|^{p-2}\nabla \tilde{u}\cdot\big(\nabla \tilde{u} -\nabla \tilde{u}_n\big)+\int_{\Omega}|\nabla \tilde{v}|^{p-2}\nabla \tilde{v} \cdot \big(\nabla \tilde{v} -\nabla \tilde{v}_n\big)\Bigg)\\
&=0.
\end{align*}

Then, recalling that $\{\tilde{u}_n\}_\mathbb{N}$ and $\{\tilde{v}_n\}_\mathbb{N}$ are both bounded in $W^{1,p}_0(\Omega)$, by Lemma \ref{monotonicityplaplacian}, we end up with
\begin{align*}
\nonumber &\limsup\limits_{n\to\infty} \bigg(\|\tilde{u}_n-\tilde{u}\|^p_{W^{1,p}_0}+\|\tilde{v}_n-\tilde{v}\|^p_{W^{1,p}_0}\bigg)\\
&\leqslant C\limsup_{k\to\infty}\Bigg(\int_{\Omega}\Big(|\nabla \tilde{u}_n|^{p-2}\nabla \tilde{u}_n-|\nabla \tilde{u}|^{p-2}\nabla \tilde{u}\Big)\cdot\nabla( \tilde{u}_n -\tilde{u})
\\
&+\int_{\Omega}\Big(|\nabla \tilde{v}_n|^{p-2}\nabla \tilde{v}_n-|\nabla \tilde{v}|^{p-2}\nabla \tilde{v} \Big)\cdot \nabla(\tilde{v}_n -\tilde{v})\Bigg) 
=0,
\end{align*}
where $C\geqslant 2+2\sup_n\{\|\nabla \tilde{u}_n\|^p_{L^p}+\|\nabla \tilde{v}_n\|^p_{L^p}\}$, and the result follows.
\end{proof}

In order to simplify some proofs, and also to let this manuscript self-contained, below we state and prove two well-known regularity results.

\begin{lemma}
\label{linfinity} Let
 $H:\Omega\times \mathbb{R}\to \mathbb{R}$ be a Carathéodory functional and $F\in L^t(\Omega)$, $t>\frac{N}{p}$. 
 \begin{itemize}
 \item[(a)] Suppose that $H(.,s)s\geqslant 0$ a.e. in $\Omega$ for all $s\in \mathbb{R}$.
 Given $w\in W^{1,p}_0(\Omega)$ a weak subsolution of
 \begin{equation}\label{1121}
 \begin{cases}
 -\Delta_p w+H(x,w)=F \mbox{ in } \Omega\\
 w=0 \mbox{ on } \partial \Omega,
 \end{cases}
 \end{equation}
 then $w\in L^\infty(\Omega)$ and $\|w\|_{L^\infty}\leqslant C \|F\|_{L^t}^{\frac{1}{p-1}}$, where $C=C(p,t,N,\Omega)>0$.
\item[(b)] Now, suppose that $|H(.,s)|\leqslant C_0 (\beta|s|^\theta+\alpha)$, where $\alpha \geqslant 0$ $\beta>0$, and $0< \theta < \frac{p^2}{N-p}$, for all $s\in \mathbb{R}$ a.e. in $\Omega$. Given $w\in W^{1,p}_0(\Omega)$ a weak subsolution of 
\begin{equation}
\label{11215}
 \begin{cases}
 -\Delta_p w=H(x,w) \mbox{ in } \Omega\\
 w=0 \mbox{ on } \partial \Omega,
 \end{cases}
 \end{equation}
\end{itemize}
\end{lemma}

then $v\in L^\infty(\Omega)$ and $\|w\|_{L^\infty}\leqslant C \big(\beta^{\frac{1}{p-1}}\|w\|^{\frac{\theta}{(p-1)}}_{L^{p^*}}+\alpha^{\frac{1}{p-1}}\big)$, where $C=C(p,C_0,N,\theta,\Omega)>0$.
\begin{proof}
Since $w \in W^{1,p}_0(\Omega)$ is a weak subsolution of \eqref{1121}, in order to prove (a) it is enough to take $G_k(w)$ as a test function in the corresponding inequality \eqref{1121} and then, remarking that
\[\int_{\Omega} H(x,w)G_k(w) \geqslant 0,\] and that $t>\frac{N}{p}$, by straightforward applications of H\"{o}lder's inequality in $F$ and $G_k(w)$, the result follows directly by Lemma 6.2 p. 49 in \cite{bocc}.

We turn our attention to the proof of (b). For this, let us consider $G_k(w)$ as a test function in the corresponding inequality of \eqref{11215} so that by the Sobolev and H\"{o}lder inequalities,  $p^*/\theta$

\begin{align*}
    \bigg(\int_\Omega |G_k(w)|^{p^*}\bigg)^{\frac{p}{p^*}}&\leqslant C \int_\Omega |\nabla G_k(w)|^p \leqslant C\int_\Omega H(x,w)
    G_k(w)
    \\
    & \leqslant C\int_{\Omega}(\beta|w|^\theta +\alpha)|G_k(w)|
    \\
    & \leqslant C\Bigg(\beta\bigg(\int_{\Omega}|w|^{p^*}\bigg)^{\frac{\theta}{p^*}}\bigg( \int_\Omega |G_k(w)|^{\frac{p^*}{p^*-\theta}}\bigg)^{\frac{p^*-\theta}{p^*}}+\alpha \int_\Omega |G_k(w)|\Bigg)
    \\
    &\leqslant C\Bigg(\bigg(\beta\int_{\Omega}|w|^{p^*}\bigg)^{\frac{\theta}{p^*}}+\alpha\Bigg)\bigg( \int_\Omega |G_k(w)|^{\frac{p^*}{p^*-\theta}}\bigg)^{\frac{p^*-\theta}{p^*}},
\end{align*}
where we have used that $p^*>\theta$ and $p^*>\frac{p^*}{p^*-\theta}$, since $p<N$, $\theta<\frac{p^2}{N-p}$ and clearly $p^*-\theta>1$. Thus,   once again by the H\"{o}lder inequality but now for $p^*-\theta$ and $\frac{p^*-\theta}{p^*-\theta-1}$ we get
\[\int_\Omega |G_k(w)|^{\frac{p^*}{p^*-\theta}}\leqslant\bigg( \int_\Omega |G_k(w)|^{p^*}\bigg)^{\frac{1}{p^*-\theta}}\mbox{meas}(\Omega_k)^{\frac{p^*-\theta-1}{p^*-\theta}},\]
where $\Omega_k=\{x\in \Omega: |w(x)|>k\}.$ Thence, by plugging the latter inequalities, we end up with
\[  \bigg(\int_\Omega |G_k(w)|^{p^*}\bigg)^{\frac{p-1}{p^*}}\leqslant C\bigg(\beta\|w\|^{\theta}_{L^{p^*}} +\alpha\bigg) \mbox{meas}(\Omega_k)^{\frac{p^*-\theta-1}{p^*}} \]

At this point, remark that given $0<k<\tau$ 
\begin{align*}(h-k)^{p-1}\mbox{meas}(\Omega_h)^{\frac{p-1}{p^*}}&\leqslant \bigg(\int_{\Omega_h}(h-k)^{p^*}\bigg)^{\frac{p-1}{p^*}}\\&\leqslant \bigg(\int_{\Omega_h} |G_k(w)|^{p^*}\bigg)^{\frac{p-1}{p^*}}
\leqslant \bigg(\int_{\Omega} |G_k(w)|^{p^*}\bigg)^{\frac{p-1}{p^*}},
\end{align*}
since $\Omega_\tau=\{x\in\Omega: w(x)>\tau\}\subset \Omega_k.$

In this fashion, by combining the last two inequalities it is straightforward to check that
\[\mbox{meas}(\Omega_h)\leqslant C\dfrac{\beta^{\frac{p^*}{p-1}}\|w\|_{L^{p^*}}^{\frac{p^*\theta}{(p-1)}}+\alpha^{\frac{p^*}{(p-1)}}}{(h-k)^{p^*}}\mbox{meas}(\Omega_k)^{\frac{p^*-\theta-1}{p-1}},  \forall \ 0<k<\tau.\]
Now, remark that since $0<\theta<\frac{p^2}{N-p}$ we have $\frac{p^*-\theta-1}{p-1}>1$ and the result follows by applying Lemme 4.1 (ii) p.93-94 in \cite{Stam3} for $\varphi(h)=\mbox{meas}(\Omega_h)$, $\alpha=p^*$ and $\beta=\frac{p^*-\theta-1}{p-1}$.
\end{proof}

	\subsection*{Nonlinear Convergence for Two Variables}
	
The next lemma is used to pass to the limit in delicate terms with low regularity. Despite that its based in the classic Vitali Theorem, we decided to formulate it in a more general form that we use in the present manuscript and to allocate it in the appendix, since it may be useful in other contexts. 
	
\begin{lemma}\label{nonlinearconvergences} Consider $\Omega \subset \mathbb{R}^N$, bounded, $H:\Omega \times \mathbb{R}\times \mathbb{R}\to \mathbb{R}$, Carath\'eodory, and $\{\tilde{u}_k\}_\mathbb{N}, \{\tilde{v}_k\}_\mathbb{N}$. Suppose that, for $i=1,2$ given $0\leqslant s_i<+\infty$, $\max\{s_i,1\}<q_i<+\infty$, $\frac{q_i}{q_i-s_i}<t_i\leqslant+\infty$, there exist $w_i\in L^{t_i}(\Omega), w_i\geqslant 0$ a.e. in $\Omega$, $\tilde{u}$ and $\tilde{u}\in L^{q_1}(\Omega)$ and $\tilde{v}\in L^{q_2}(\Omega)$, for which
\begin{itemize}
\item[(i)] $\tilde{u}_k\to \tilde{u}$ and $\tilde{v}_k\to \tilde{v}$, a.e. in $\Omega$, if $k\to +\infty$,
\item[(ii)] $\{\tilde{u}_k\}_\mathbb{N}$ is bounded in $L^{q_1}(\Omega)$ and $\{\tilde{v}_k\}_\mathbb{N}$ is bounded in $L^{q_2}(\Omega)$, 
\item[(iii)] $\int_{\{|\tilde{u}_k|>n\}} |H(x,\tilde{u}_k,\tilde{v}_k)| \leqslant \int_{\{|\tilde{u}_k|>n\}} w_1 |\tilde{u}_k|^{s_1}$ for all $k,n\in \mathbb{N}$,
\item[(iv)] Given $n_0\in \mathbb{N}$ there exists $C=C(n_0)>0$, such that $$\int_{\{|\tilde{u}_k|\leqslant n_0\}} |H(x,\tilde{u}_k,\tilde{v}_k)| \leqslant C(n_0)\int_{\{|\tilde{u}_k|\leqslant n_0\}} w_2 |\tilde{v}_k|^{s_2}$$ for all $k\in \mathbb{N}$.
\end{itemize}
Then $H(., \tilde{u}_k,\tilde{v}_k)\to H(., \tilde{u},\tilde{v})$ in $L^1(\Omega)$ if $k\to+\infty.$
    \end{lemma}
\begin{proof}
Since $H$ is Carath\'eodory, from (i) we already have that
$H(x,\tilde{u}_k,\tilde{v}_k)\to H(x,\tilde{u},\tilde{v})$ a.e. in $\Omega$. Thus, it is enough to prove that $\{H(x,\tilde{u}_k,\tilde{v}_k)\}_\mathbb{N}$ is uniformly integrable.

Indeed, consider $E\subset \Omega$, measurable,  and $n\in \mathbb{N}$. By (iii), it is clear that
\begin{equation}
\label{1234}
\int_E|H(x,\tilde{u}_k,\tilde{v}_k)| \leqslant \int_{E \cap |\tilde{u}_k|\leqslant n\}}|H(x,\tilde{u}_k,\tilde{v}_k)| + \int_{\{|\tilde{u}_k|> n\}} w_1 |\tilde{u}_k|^{s_1}.
\end{equation}
Next, recall that by the Chebyshev inequality, 
\[\mbox{meas}(\{|\tilde{u}_k|> n\})\leqslant \dfrac{\|u_k\|^{q_1}}{n^{q_1}}, \ \forall k,n\in \mathbb{N}.\]
In this manner, if $t_1=+\infty$. by the H\"older inequality for $\frac{q_1}{s_1}$ and $\frac{q_1}{q_1-s_1}$ there follows
\begin{align*}
\int_{\{|\tilde{u}_k|> n\}}w_1|\tilde{u}_k|^{s_1} &\leqslant \|w_1\|_{L^{\infty}} \|\tilde{u}_k\|_{L^{q_1}}^{s_1}\big(\mbox{meas}(\{|\tilde{u}_k|> n\})\big)^{\frac{q_1-s_1}{q_1}}\\
&\leqslant \dfrac{1}{n^{q_1-s_1}} \|\tilde{u}_k\|_{L^{q_1}}^{s_1}\|w_1\|_{L^{\infty}},.
\end{align*}

Now, when $t_i$ is finite,  by the the H\"older inequality for $t_1$, $\frac{t_1 q_1}{t_1(q_1-s_1)-q_1}$, $\frac{q_1}{s_1}$, if $s_1>0$, or $t_1$, $\frac{t_1}{t_1-1}$, if $s_1=0$, we arrive at
\begin{align*}
\int_{\{|\tilde{u}_k|> n\}}w_1|\tilde{u}_k|^{s_1} &\leqslant \|w_1\|_{L^{t_1}} \|\tilde{u}_k\|_{L^{q_1}}^{s_1}\big(\mbox{meas}(\{|\tilde{u}_k|> n\})\big)^{\frac{t_1(q_1-s_1)-q_1}{t_1 q_1}}\\
&\leqslant \dfrac{C}{n^{\frac{q_1(t_1-1)}{t_1}-s_1}} \|w_1\|_{L^{t_1}}, 
\end{align*}
where $C=\sup_{k} \{ \|\tilde{u}_k\|_{L^{q_1}}^{\frac{q_1}{t_1^\prime}}\}$ is finite by (ii).

Hence, by the last two inequalities since $q_1>s_1$ and $\frac{q_1(t_1-1)}{t_1}>s_1$, given $\varepsilon>0$, consider $n_0\in\mathbb{N}$ satisfying
\begin{equation}
\label{1301}
\int_{\{|\tilde{u}_k|> n\}}w_1|\tilde{u}_k|^{s_1}<\frac{\varepsilon}{2}.
\end{equation}
Further, considering this $n_0$ fixed, by (iv) we have that
\[\int_{E\cap \{|\tilde{u}_k|\leqslant n_0\}} |H(x,\tilde{u}_k,\tilde{v}_k)| \leqslant C(n_0)\int_{\{|\tilde{u}_k|\leqslant n_0\}} w_2 |\tilde{v}_k|^{s_2}.\]

In particular, if $t_2$ is finite, by the H\"older inequality, for $t_2$, $\frac{t_2}{t_2-1}$, if $s_2=0$, or  $t_2$, $\frac{t_2 q_2}{t_2(q_2-s_2)-q_2}$, $\frac{q_2}{s_2}$, if $s_2>0$, we get

\begin{align*}
\int_{E\cap \{|\tilde{u}_k|\leqslant n_0\}} | H(x,\tilde{u}_k,\tilde{v}_k)| &\leqslant C(n_0)\|w_2\|_{L^{t_2}} \|\tilde{v}_k\|_{L^{q_2}}^{s_2} \big(\mbox{meas}(E)\big)^{\frac{t_2(q_2-s_2)-q_2}{t_2 q_2}}\\
&\leqslant C_2 \|w_2\|_{L^{t_2}}\big(\mbox{meas}(E)\big)^{\frac{t_2(q_2-s_2)-q_2}{t_2 q_2}},
\end{align*}
for all $k \in \mathbb{N}$, where $C_2=C(n_0)\sup_k \|\tilde{v}_k\|_{L^{q_2}}^{s_2}<+\infty$ by (ii). Moreover, if $t_2=+\infty$, by the H\"older inequality for $\frac{q_2}{s_2}$ and $\frac{q_2}{q_2-s_2}$ there follows 
\begin{align*}
\int_{E\cap \{|\tilde{u}_k|\leqslant n_0\}} | H(x,\tilde{u}_k,\tilde{v}_k)| &\leqslant C(n_0)\|w_2\|_{L^{\infty}} \|\tilde{v}_k\|_{L^{q_2}}^{s_2} \big(\mbox{meas}(E)\big)^{\frac{q_2-s_2}{q_2}}\\
&\leqslant C_2 \|w_2\|_{L^{t_2}}\big(\mbox{meas}(E)\big)^{\frac{q_2-s_2}{q_2}}.
\end{align*}
In this way, by using that $q_2>s_2$ and $t_2(q_2-s_2)>q_2$, consider $\delta>0$ such that if $\mbox{meas}(E)<\delta$, then
\begin{equation}
\label{1423}
\int_{E\cap \{|\tilde{u}_k|\leqslant n_0\}} |H(x,\tilde{u}_k,\tilde{v}_k)| < \frac{\varepsilon}{2}.
\end{equation}
Thence, by combining \eqref{1234},\eqref{1301} and \eqref{1423}, for every $E\subset \Omega$, measurable, for which $\mbox{meas}(E)<\delta$, there follows that
\[\int_E |H(x,\tilde{u}_k,\tilde{v}_k)| < \varepsilon, \forall k\in \mathbb{N},\]
so that $\{H(.,\tilde{u}_k,\tilde{v}_k)\}_\mathbb{N}$ is uniformly integrable and the result follows by the Vitali Theorem. 
\end{proof}

\section*{Acknowlegments}

The first author was partially supported by FEMAT/Brazil grant 01/2023 and Programa de Educa\c c\~ao Tutorial/FNDE/CAPES/Brazil grant 0331131595. The second author was partially supported  by CAPES/Brazil grant 88887.480897/2020-00,  CNPq/Brazil grant 142279/2020-0 and FAPEMIG/Brazil grant APQ-04528-22.

\bibliographystyle{plain}

\end{document}